\documentclass[12pt]{article}
\usepackage{amsfonts}
\usepackage{amssymb}
\usepackage{mathrsfs}
\usepackage{srcltx}
\usepackage{amsmath}
\usepackage{amsthm}
\usepackage{amstext}
\usepackage{amsopn}
\usepackage{graphicx}
\usepackage{subcaption}
\usepackage{grffile}
\usepackage{float}
\usepackage{epsfig}
\usepackage{overpic}
\usepackage{epstopdf}
\usepackage{color}
\usepackage{cite}
\usepackage{graphicx}
\usepackage{multirow}
\usepackage[pagewise]{lineno}
\graphicspath{{./figs/}}
\DeclareGraphicsExtensions{.eps}

\newtheorem{theorem}{Theorem}[section]
\newtheorem{lemma}[theorem]{Lemma}

\theoremstyle{definition}

\allowdisplaybreaks[4]

\numberwithin{equation}{section}
\theoremstyle{remark}
\newtheorem{remark}[theorem]{Remark}

\numberwithin{equation}{section}

\newcommand{\ba}{\begin{array}}
	\newcommand{\ea}{\end{array}}

\begin{document}
	\date{}
	\title{ \bf\large{Grazing duration and intensity modulate vegetation dynamics in semi-arid ecosystems with seasonal succession}\footnote{This work was partially supported by grants from National Science Foundation of China (12371503, 12071382), Natural Science Foundation of
			Chongqing (CSTB2024NSCQ-MSX0992).}}
	\author{ Junhong Gan,\ \ Guohong Zhang\footnote{Corresponding Author. Email: zgh711@swu.edu.cn  },\ \
		 Xiaoli Wang\\
		{\small School of Mathematics and Statistics, Southwest
			University, }\\
		{\small Chongqing, 400715, P.R. China.} }
	\maketitle
	\noindent
	\begin{abstract}
		{
			This study investigates the impacts of grazing duration and intensity on vegetation population dynamics in semi-arid ecosystems characterized by seasonal succession. A novel piecewise periodic model is proposed, dividing the annual cycle into three distinct phases: dry season, growth period and grazing period in wet season. We derive critical thresholds for the durations of the dry season and grazing period that determine the persistence or extinction of a single vegetation species. For two competing species, we analyze how grazing parameters influence competitive outcomes, including exclusion, coexistence, and bistability. Theoretical results are supported by numerical simulations, which illustrate bifurcation diagrams and phase transitions under varying grazing regimes. Our findings provide actionable insights for sustainable grazing management in arid and semi-arid regions.
		}
		
	\end{abstract}
	
	\section{Introduction}
	\ \  \ \
	In recent decades, desertification and climate change have posed severe threats to the stability of vegetation ecosystems in arid and semi-arid regions. The dynamics of these ecosystems are influenced by a complex interplay of climatic factors, soil moisture availability, and grazing.
Previous studies have largely focused on reaction-diffusion models to explain the emergence and evolution of spatial patterns in vegetation systems (see \cite{vegetation1, vegetation2,vegetation4,vegetation3,vegetation5,vegetation6,vegetation7} and reference therein). At the same time, in order to accounting for  the seasonality of the climate in some semi-arid regions, that is wet and dry  season, integro-difference models are employed  to  account for the temporal separation
	of plant growth processes during the wet season and seed dispersal processes during
	the dry season \cite{vegetation11}. An
impulsive reaction-advection-diffusion model is proposed to analytically investigate
the effects of rainfall intermittency on the onset of patterns\cite{ vegetation12}.

While rainfall has long been recognized as a primary driver of vegetation pattern formation, the role of grazing intensity and its interaction with plant-water feedback mechanisms remain important\cite{vgrazing1, vgrazing3,vgrazing2,Harvest1,Harvest2,Harvest3,Harvest5,Harvest6,Harvest7, Harvestsun}.
 In particular, Metzger et al. \cite{seasonal} used a spatial ecosystem model to compare two grazing areas: one is grazed only during the wet season, and the other is grazed year around year. They found that the two areas were similar with respect to grazer density during the wet season but not in the dry season.

	Despite advances above, existing models often oversimplify seasonal dynamics by considering only two phases, that is  wet and dry season. However, in reality, semi-arid ecosystems often exhibit three distinct phases: a dry season with high plant mortality, a growth period without grazing and a grazing period within the wet season.  Thereby existing models often fail to isolate the specific effects of grazing duration and intensity within the wet season. This is particularly relevant for semi-arid ecosystems, where grazing often occurs only during a portion of the wet season.
This seasonal succession profoundly influences plant competition, persistence, and community structure. Understanding how the timing and intensity of grazing influence vegetation competition and persistence is critical for designing sustainable grazing management systems.

To address the gap above, we assume there exist three distinct time periods in semi-arid vegetation system (see Figure 1):
(1) a dry season;
(2) a growth period without grazing within wet season;
(3) a grazing period within  wet season.
 At the same time, we ignore the effect of spatial mobility for vegetation. Then a novel three-stage Lotka-Volterra competition model is proposed as following:
	\begin{equation}\label{1.5}
		\begin{cases}
			\displaystyle  \frac{du}{dt} = -d_1 u,  &\;\;  t \in (nT, nT + \tau_1],\\
			\displaystyle  \frac{dv}{dt} = -d_2 v,  &\;\;  t \in (nT, nT + \tau_1],
		\end{cases}
	\end{equation}
	
	\begin{equation}\label{1.6}
		\begin{cases}
			\displaystyle \frac{du}{dt} = r_1 u (1 - \frac{u}{K_1} - b_1 v),&\;\;  t \in (nT + \tau_1, nT + \tau_2],\\
			\displaystyle \frac{dv}{dt} = r_2 v (1 - \frac{v}{K_2} - b_2 u),&\;\; t \in (nT + \tau_1, nT + \tau_2],
		\end{cases}
	\end{equation}
	
	\begin{equation}\label{1.7}
		\begin{cases}
			\displaystyle \frac{du}{dt} = r_1 u (1 - \frac{u}{K_1} - b_1 v) - q_{1}E_{1} u, &\;\;  t \in (nT + \tau_2, (n+1) T],\\
			\displaystyle \frac{dv}{dt} = r_2 v (1 - \frac{v}{K_2} - b_2 u) - q_{2}E_{2} v,&\;\; t \in (nT + \tau_2, (n+1) T].
		\end{cases}
	\end{equation}
	Here $u(t)$ and $v(t)$ represent densities of two types of vegetation populations. In bad season, $d_1$ and $d_2$ describe  the natural mortality rates of species $u(t)$ and $v(t)$, respectively. In wet season, $r_1$ and $r_2$ are net birth rates, $K_1$ and $K_2$ are the carrying capacities,  $b_1$ and $b_2$ are the competition coefficients of species $u(t)$ and $v(t)$, respectively.  $\tau_1$ is duration  of the dry season, $\tau_2-\tau_1$ represents the length of growth period and  $T-\tau_2$ is the duration of grazing period for vegetation (see Figure\ref{seasonalchart}). In the following sections, we always assume that $\tau_2\geq\tau_1$. $q_{1}E_{1}$ and $q_{2}E_{2}$ are the grazing intensities of species $u(t)$ and $v(t)$, respectively.
We point out that the structure of model  \eqref{1.5}-\eqref{1.7} allows us to independently analyze the impacts of grazing duration and intensity on vegetation dynamics, which is a refinement over previous two-stage models.

	\begin{figure}[htbp]
		\centering
		\includegraphics[width=\textwidth]{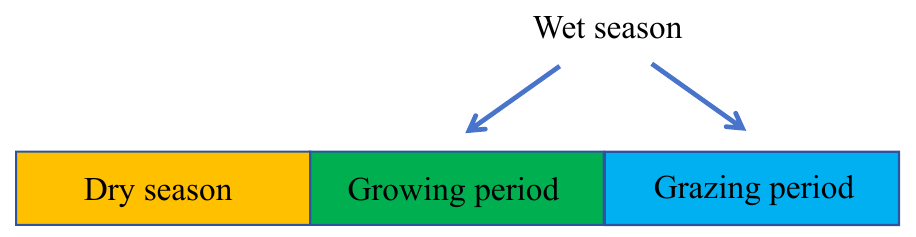}
		\captionsetup{font={footnotesize}}
		\caption{ Seasonal succession with three distinct time periods.  The duration  of the bad (unfavorable/dry) season is $\tau_1$, $\tau_2-\tau_1$ represents the length of growing period and  $T-\tau_2$ is the grazing period for vegetation in good (favorable) season.}
		\label{seasonalchart}
	\end{figure}	
In recent years, some time-periodic mathematical models with seasonal succession  have been developed to describe the periodic environment factors in population growth and disease transmission systems \cite{succession1,succession2}.   Piecewise continuous functions, rather than continuous periodic functions, have been employed to describe these seasonally switching systems \cite{succession3,succession4,succession5,SHU2025}. Particularly, Hsu and Zhao \cite{Hsu.Zhao} investigated a two-species competition model incorporating seasonal succession and give a complete classification for the global dynamics of the competition  system.
	Following  \cite{Hsu.Zhao}, Liu et al. \cite{Liu.2cs} extended the model by introducing distinct harvesting strategies and investigated the impact of the proportion of good season on species competition outcomes and harvesting practices.
	At the same time, a series of fishery models with closed seasons and open seasons have been investigated\cite{Liu.MMHarvest1,Liu.MMHarvest2}.
	Similar modeling approach has also been applied to the mathematical modeling of pulse and periodic release strategies for mosquito populations\cite{Yu.mosquito1, Yu.mosquito2,Yu.mosquito3,Yu.mosquito4}.

The main purpose of this paper is to explore  how the duration of grazing period and intensity of grazing affect the persistence and extinction  of the single vegetation species,  and extend the analysis to a two-species competition system, identifying how grazing parameters regulate competitive outcomes (exclusion, coexistence, or bistability). In particular, we investigate the existence and stability of positive periodic solutions, which are of great significance for understanding the long-term coexistence of species in seasonal environments.
This work not only extends the theoretical framework of seasonal succession models but also offers practical insights for grassland management in semi-arid regions such as northwestern China.

	For simplicity, we let $\bar{u} = \frac{u}{K_1}$, $\bar{t} = r_1 t$, $r = \frac{r_2}{r_1}$, $\bar{b}_1 = b_1 K_2$, $\bar{b}_2 = b_2 K_1$, $\bar{c}_1 = \frac{q_{1}E_{1}}{r_1}$, $\bar{c}_2 = \frac{r_2 q_{2}E_{2}}{c_1}$. We then rescale system \eqref{1.5}-\eqref{1.7}, drop the hat sign and obtain
	\begin{equation}\label{main1.1}
		\begin{cases}
			\displaystyle  \frac{du}{dt} = -d_1 u,  &\;\;  t \in (nT, nT + \tau_1],\\
			\displaystyle  \frac{dv}{dt} = -d_2 v,  &\;\;  t \in (nT, nT + \tau_1],
		\end{cases}
	\end{equation}
	
	\begin{equation}\label{main1.2}
		\begin{cases}
			\displaystyle \frac{du}{dt} = u(1 - u - b_1 v),&\;\;  t \in (nT + \tau_1, nT + \tau_2],\\
			\displaystyle \frac{dv}{dt} = r v(1 - v - b_2 u),&\;\; t \in (nT + \tau_1, nT + \tau_2],
		\end{cases}
	\end{equation}
	
	\begin{equation}\label{main1.3}
		\begin{cases}
			\displaystyle \frac{du}{dt} = u (1 - u - b_1 v) - c_1 u, &\;\;  t \in (nT + \tau_2, (n+1) T],\\
			\displaystyle \frac{dv}{dt} = r v (1 - v - b_2 u) - c_2 v,&\;\; t \in (nT + \tau_2, (n+1) T].
		\end{cases}
	\end{equation}

	This paper is organized as follows. In Section 2, we focus on the global dynamics of single vegetation population with seasonal grazing model, which leads to the existence conditions of semi-trivial steady states. In Section 3, we provide an understanding of the global dynamics of competition system \eqref{main1.1}-\eqref{main1.3}. In Section 4, we present numerical examples that serve to confirm our theoretical results. In Section5, we offer a  discussion of our findings.
	
	\section{Single-species persistence and extinction under seasonal grazing}
	\subsection{Preliminary lemmas and notation}
	\ \ \ \ \ System \eqref{main1.1}-\eqref{main1.3} admits a trivial steady state $E_0 = (0,0)$. To study the semi-trivial steady states of the system \eqref{main1.1}-\eqref{main1.3}, it is necessary to investigate the following two subsystems
	\begin{equation}\label{main1.4}
		\begin{cases}
			\displaystyle \frac{du}{dt} = -d_1 u, &\;\;  t \in (nT, nT + \tau_1],\\
			\displaystyle \frac{du}{dt} = u (1 - u),&\;\; t \in (nT + \tau_1, nT + \tau_2],\\
			\displaystyle \frac{du}{dt} = u (1 - u) - c_1 u,&\;\; t \in (nT + \tau_2, (n+1)T],\\
		\end{cases}
	\end{equation}
	and
	\begin{equation}\label{main1.5}
		\begin{cases}
			\displaystyle \frac{dv}{dt} = -d_2 v, &\;\;  t \in (nT, nT + \tau_1],\\
			\displaystyle \frac{dv}{dt} = r v (1 - v),&\;\; t \in (nT + \tau_1, nT + \tau_2],\\
			\displaystyle \frac{dv}{dt} = r v (1 - v) - c_2 v,&\;\; t \in (nT + \tau_2, (n+1)T].\\
		\end{cases}
	\end{equation}
	For system \eqref{main1.4}, we define
	\begin{align*}
		\overline{H}(x) &= u(\tau_1; 0, x), & \widetilde{H}(x) &= u(\tau_2; 0, x),  &  H(x) &= u(T; 0, x), \\
		\overline{H}_n(x) &= u(nT+\tau_1; 0, x), & \widetilde{H}_n(x) &= u(nT+\tau_2; 0, x), & H_n(x) &= u(nT;0,x),
	\end{align*}
	\(n=0,1,2,\cdots\), for an initial value $x > 0$ such that $\overline{H}_0(x)=\overline{H}(x), \widetilde{H}_0(x)=\widetilde{H}(x)$ and $ H_0(x) = x. $ Functions $ \overline{H}(x) , \widetilde{H}(x)$ and $H(x)$ are continuously differentiable.
	Similarly, for system \eqref{main1.5}, we define
	\begin{align*}
		\overline{K}(y) &= u(\tau_1; 0, y), & \widetilde{K}(y) &= u(\tau_2; 0, y),  &  K(y) &= u(T; 0, y), \\
		\overline{K}_n(y) &= u(nT+\tau_1; 0, y), & \widetilde{K}_n(y) &= u(nT+\tau_2; 0, y), & K_n(y) &= u(nT;0,y),
	\end{align*}
	$n=0,1,2,\cdots$, for an initial value $y > 0$ such that $\overline{K}_0(y) = \overline{K}(y), \widetilde{K}_0(y)=\widetilde{K}(y)$ and $ K_0(y) = y.$ Functions $ \overline{K}(y) , \widetilde{K}(y)$ and $K(y)$ are continuously differentiable.
	
	We first introduce a fundamental lemma  that is an extension of Lemma 2.1 of \cite{Liu.2cs} and serves as a basis for the proofs of our results in this paper.
	
	\begin{lemma}\label{lemma2.1}
		For any given initial values $x>0$ and $y>0$,  the following statements are valid.
		
		$(\mathbf{i})$ If $H(x)>x$, then sequences $\left\{ \overline{H}_{n}(x) \right\}_{0}^{\infty}$ , $\left\{ \widetilde{H}_{n}(x) \right\}_{0}^{\infty}$ and $\left\{ H_{n}(x) \right\}_{0}^{\infty}$ are strictly increasing. If $K(y)>y$, then sequences $\left\{ \overline{K}_{n}(y) \right\}_{0}^{\infty}$ , $\left\{ \widetilde{K}_{n}(y) \right\}_{0}^{\infty}$ and $\left\{ K_{n}(y) \right\}_{0}^{\infty}$ are strictly increasing.
		
		$(\mathbf{ii})$ If $H(x)=x$, then $H_{n}(x) = x$ for $n = 0, 1, 2, \cdots$. Furthermore, $u(t; 0, x)$ is a $T$-periodic solution of system \eqref{main1.4}. If $K(y) = y $, then $K_{n}(y) = y$ for $ n = 0, 1, 2, \cdots$. Furthermore, $v(t; 0, y)$ is a $T$-periodic solution of system \eqref{main1.5}.
		
		$(\mathbf{iii})$ If $H(x)<x$, then sequences $\left\{ \overline{H}_{n}(x) \right\}_{0}^{\infty}$, $\left\{ \widetilde{H}_{n}(x) \right\}_{0}^{\infty}$ and $\left\{ H_{n}(x) \right\}_{0}^{\infty}$ are strictly decreasing. If $ K(y) < y $, then sequences $\left\{ \overline{K}_{n}(y) \right\}_{0}^{\infty}$ , $\left\{ \widetilde{K}_{n}(y) \right\}_{0}^{\infty}$ and $\left\{ K_{n}(y) \right\}_{0}^{\infty}$ are strictly decreasing.
		
		$(\mathbf{iv})$ $\lim_{n \to \infty} H_{n}(x) = p_{1}$, where $H(p_{1}) = p_{1}$ and  $\lim_{n \to \infty} K_{n}(y) = p_{2}$, where $K(p_{2}) = p_{2}$.
	\end{lemma}
	\begin{proof}
		$(\mathbf{i})$ 	
		Suppose $u(t)=u(t;0,x)$ is a solution of system \eqref{main1.4}. By induction, we have
		\begin{equation*}
			H_0(x)=x, H_{n+1}(x)=u(T;0,H_n(x))=H(H_n(x)),
		\end{equation*}
		\begin{equation*}
			\overline{H}_0(x)=u(\tau_1;0,x), \overline{H}_n(x)=\overline{H}(H_n(x))
		\end{equation*}
		and
		\begin{equation*}
			\widetilde{H}_0(x)=u(\tau_2;0,x), \widetilde{H}_n(x)=\widetilde{H}(H_n(x)), n=1,2,3,\cdots.
		\end{equation*}
		If $H(x)>x$, we can show $H_{i+1}(x)>H_{i}(x)$, for $i=1,2,3,\cdots$. In fact, if it is not true, then there exists positive integer $n_0$ such that  $H_{i+1}(x)>H_{i}(x)$, for $i=1,2,\cdots,n_0-1$, but $H_{n_{0}}(x)\geq H_{n_0+1}(x)$. Write $\bar{u}(t):=u(T+t;0,x).$ Since system \eqref{main1.4} is $T$-periodic, $\bar{u}(t)$ is also solution of system \eqref{main1.4}. Noting that $\bar{u}(0)=u(T;0,x)=u(T)=H(x)$, we can rewrite $\bar{u}(t)=u(t;0,H(x))$. Comparing $u(t)$ with $\bar{u}(t)$, we have
		\begin{equation*}
			u((n_0-1)T;0,x)=H_{n_0-1}(x)<H_{n_0}(x)=u(n_0T;0,x)=\bar{u}((n_0-1)T)
		\end{equation*}
		\begin{equation*}
			u(n_0T;0,x)=H_{n_0}(x)\geq H_{n_0+1}(x)=u((n_0+1)T;0,x)=\bar{u}(n_0T).
		\end{equation*}
		Hence $u(t)$ with $\bar{u}(t)$ must intersect on $((n_0-1)T,n_0T]$. However, in an autonomous system, solutions starting from different initial values cannot intersect, which leads to a contradiction. We have proved that $\left\{ H_{n}(x) \right\}_{0}^{\infty}$ is strictly increasing.
		
		Now we can show that $\overline{H}_{i+1}(x)>\overline{H}_{i}(x)$ for $i=0,1,2,\cdots$. In fact, if it is not true, then there exists positive integer $n_0$ such that  $\overline{H}_{n_0}(x)\geq \overline{H}_{n_0+1}(x)$, which implies that   $u(n_0T+\tau_1)\geq u((n_0+1)T+\tau_1)$. After $T-\tau_1$, it becomes $u((n_0+1)T)\geq u((n_0+2)T)$, i.e. $H_{n_0+1}(x)\geq H_{n_0+2}(x)$, which is in contradiction with the monotonicity of $\left\{ H_{n}(x) \right\}_{0}^{\infty}$. Similarly, we can prove $\left\{ \widetilde{H}_{n}(x) \right\}_{0}^{\infty}$ is strictly increasing.
		Similarly, we can give the proof of  $(\mathbf{ii})$ and $(\mathbf{iii})$.
		
		$(\mathbf{iv})$ If $H(x)=x$, then conclusion follows immediately.
		Suppose $H(x)>x$. Then from $(\mathbf{i})$, sequence  $\left\{ H_{n}(x) \right\}_{0}^{\infty}$ is strictly increasing. We have $0<x<1$ and $0<H_n(x)<1,n=0,1,2,\cdots$. Hence,  $\left\{ H_{n}(x) \right\}_{0}^{\infty}$ is as monotonically increasing sequence with an upper bound. Then there exists $p_1 \geq 0$ such that $\lim_{n \to \infty} H_{n}(x) = p_{1}$. Since $H_{n+1}(x)=H(H_n(x))$, taking the limit of both sides, we have $H(p_1)=p_1$.
		The proof for the case of $H(x)<x$ can be done similarly. 	
	\end{proof}
	In order to obtain the existence of positive periodic solution for system \eqref{main1.4}, it is necessary to investigate the properties for $H(x)$.
	\begin{lemma} \label{lem.tao2}
		Assume that initial value $x>0$.  Then we have
		\begin{equation*}
			\lim_{x \to 0} \frac{H(x)}{x} = e^{c_1(\tau_2-\tau_2^*)},
		\end{equation*}
		where $\tau_2^*:=\frac{(d_1+1)\tau_1+(c_1-1)T}{c_1}.$
	\end{lemma}
	\begin{proof}
		We solve the first equation of \eqref{main1.4} with $u(0)=x$. It follows from
		\begin{equation*}
			\frac{du}{dt} = -d_1 u
		\end{equation*}
		that
		\begin{equation} \label{lm1.1}
			d(\ln{|u|})=-d_{1}dt.
		\end{equation}		
		\noindent  Integrating \eqref{lm1.1} from $0$ to $\tau_1$ yields
		\begin{equation*}
			u(\tau_1)=x \cdot e^{-d_1 \tau_1}.
		\end{equation*}
		\noindent  Let $m_1 := e^{-d_1 \tau_1}.$
		We have
		\begin{equation} \label{lm1.2}
			\overline{H}(x)= m_1 x,
		\end{equation}
		which indicates that $\lim_{x \to 0}\frac{\overline{H}(x)}{x}=m_1.$
		
		Then we can solve the second equation of \eqref{main1.4} with $u(\tau_1^+)=u(\tau_1)$. It follows from
		\begin{equation*}
			\frac{du}{dt} = u(1-u)
		\end{equation*}
		that
		\begin{equation} \label{lm1.3}
			d(\ln{|\frac{u}{1-u}|})=dt.
		\end{equation}
		
		\noindent  Integrating \eqref{lm1.3} from $\tau_1$ to $\tau_2$  yields
		\begin{equation*}
			\frac{u(\tau_2)}{1-u(\tau_2)} = \frac{u(\tau_1)}{1-u(\tau_1)} \cdot e^{\tau_2-\tau_1}.			
		\end{equation*}
		\noindent  Let $m_2 := e^{\tau_2-\tau_1}$ and we have
		\begin{equation} \label{lm1.4}
			\frac{\widetilde{H}(x)}{1-\widetilde{H}(x)}=\frac{\overline{H}(x)}{1-\overline{H}(x)} \cdot m_2.
		\end{equation}	
		
		We continue to solve the third equation of \eqref{main1.4} with $u(\tau_2^+)=u(\tau_2)$.  If $c_1 \neq 1$, it follows that
		\begin{equation} \label{lm1.5}
			(\frac{1}{u}-\frac{1}{1-u-c_1})du=(1-c_1)dt.
		\end{equation}
		Integrating \eqref{lm1.5} from $\tau_2$ to $T$  yields
		\begin{equation*}
			\frac{u(T)}{1-u(T)-c_1}=\frac{u(\tau_2)}{1-u(\tau_2)-c_1} \cdot e^{(1-c_1)(T-\tau_2)},			
		\end{equation*}
		witch implies that
		\begin{equation} \label{lm1.6}
			\frac{H(x)}{1-H(x)-c_1} = \frac{\widetilde{H}(x)}{1-\widetilde{H}(x)-c_1} \cdot e^{(1-c_1)(T-\tau_2)}.
		\end{equation}
		We have
		\begin{equation*}
			\lim_{x \to 0}\frac{H(x)}{\widetilde{H}(x)}=e^{(1-c_1)(T-\tau_2)}.			
		\end{equation*}
		Hence,
		\begin{equation}
			\begin{aligned}  \label{lm1.7}
				\lim_{x \to 0} \frac{H(x)}{x} &= \lim_{x \to 0} \frac{H(x)}{\widetilde{H}(x)} \cdot \frac{\widetilde{H}(x)}{\overline{H}(x)} \cdot \frac{\overline{H}(x)}{x} \\
				& = e^{(1-c_1)(T-\tau_2)+\tau_2-\tau_1-d_1\tau_1} \\
				&= e^{c_1(\tau_2-\tau_2^*)},
			\end{aligned}
		\end{equation}
		where
		\begin{equation} \label{tao2}
			\tau_2^*=\frac{(d_1+1)\tau_1+(c_1-1)T}{c_1}.
		\end{equation}
		
		If $c_1=1$ , we have $\frac{du}{dt}=-u^2$. Then
		\begin{equation} \label{lm1.8}
			d(\frac{1}{u}) = dt.
		\end{equation}
		
		\noindent  Integrating \eqref{lm1.8} from $\tau_2$ to $T$  yields
		\begin{equation} \label{lm1.9}
			\frac{1}{H(x)} - \frac{1}{\widetilde{H}(x)} = T-\tau_2.
		\end{equation}
		
		\noindent  Hence,
		\begin{equation} \label{lm1.10}
			\lim_{x \to 0} \frac{H(x)}{x} = \lim_{x \to 0} \frac{H(x)}{\widetilde{H}(x)} \cdot \frac{\widetilde{H}(x)}{\overline{H}(x)} \cdot \frac{\overline{H}(x)}{x} = e^{\tau_2-\tau_1-d_1\tau_1}=e^{\tau_2-(d_1+1)\tau_1}.
		\end{equation}
	\end{proof}
	
	\begin{lemma}\label{lem.H<x}
		If  the initial value $x \geq 1$, then $H(x) < x$.
	\end{lemma}
	\begin{proof}
		It follows  from the first equation of system \eqref{main1.4} that $\overline{H}(x) < x $ for any initial value $ x\geq 1$. Note that both the second and third  equation of system \eqref{main1.4} are Logistic-type differential equation with growth rate $1$ and $1-c_{1}$, respectively. We can obtain that $H(x)<x$ for any the initial value $x \geq 1$.
	\end{proof}
	\subsection{Existence and uniqueness of $T$-periodic solution}
 \ \ \ \ In this subsection, we discuss the existence and stability of positive $T$-periodic solution, which indicates that species $u$ or $v$ can survive.
	\begin{theorem}\label{lem.u.solu}
		Assume that $\tau_1 < \tau^{*}_1:=\frac{T}{d_{1}+1}$. If $ \tau_2 > \tau_2^*$, then system \eqref{main1.4} has a unique positive  globally asymptotically stable $T$-periodic solution.
	\end{theorem}
	\begin{proof}
		$(\mathbf{i})$
		When $c_1 \neq 1$,  it follows from \eqref{lm1.7} and $\tau_2 > \tau_2^*$ that $\lim_{x \to 0} \frac{H(x)}{x} > 1$. There exists a sufficiently small $ \varepsilon > 0$ such that
		\begin{equation*}
			H(x) > x  \text{  for } x \in (0, \varepsilon).
		\end{equation*}
		From lemma \ref{lem.H<x}, we know that when $x \geq 1 $, $ H(x) < x$. Hence, there exists $x_0 \in [0, \varepsilon) $ such that $H(x_0) = x_0$,
		which indicates that $u(t; 0, x_0)$ is a $T$-periodic solution of system \eqref{main1.4}.
		
		Now we prove the uniqueness of the $T$-periodic solution of system \eqref{main1.4}.
		Assume that system \eqref{main1.4} has another $T$-periodic solution $x_1$, i.e. $H(x_1) = x_1$. Note that $\lim_{x \to 0} \frac{H(x)}{x} > 1$ when $\tau_2 > \tau_2^*$, and $ H(x) < x$ when $ x \geq 1 $.
		Without loss of generality, we assume $0<x_0<x_1<1$, where $x_0$ is the smallest fixed point of $H(x)$ and $x_1$ is the largest fixed point of $H(x)$ such that
		\begin{equation*}
			H(x) > x, \ x \in (0, x_0)\ \text{ and }\ H(x) < x , \ x \in (x_1, 1).
		\end{equation*}
		\noindent
		Then there must exist $x_2 \in [x_0, x_1] $ satisfying the following conditions.
		
		\noindent
		\textbf{Case 1:}  $ x_0 < x_2 < x_1 $,
		\begin{equation} \label{lm3.1}
			H'(x_0) = 1 , \ H'(x_2) = 1 , \ H'(x_1) < 1,
		\end{equation}
		or \begin{equation} \label{lm3.2}
			H'(x_0) = 1 ,\ H'(x_2) < 1 ,\ H'(x_1) = 1,
		\end{equation}
		or \begin{equation} \label{lm3.3}
			H'(x_0) < 1 ,\ H'(x_2) > 1 ,\ H'(x_1) < 1,
		\end{equation}
		or \begin{equation} \label{lm3.4}
			H'(x_0) < 1 ,\ H'(x_2) = 1 ,\ H'(x_1) = 1,
		\end{equation}
		
		\noindent  \textbf{Case 2:} $x_2 = x_0$,
		\begin{equation} \label{lm3.5}
			\ H'(x_0) =H'(x_2)= 1 ,\ H'(x_1) < 1,
		\end{equation}
		or \begin{equation} \label{lm3.6}
			\ H'(x_0)= H'(x_2)< 1 ,\ H'(x_1) = 1,
		\end{equation}
		
		\noindent  \textbf{Case 3:} $x_2 = x_1$,
		\begin{equation} \label{lm3.7}
			\ H'(x_0) = 1 ,\ H'(x_1)=H'(x_2) < 1,
		\end{equation}
		or \begin{equation} \label{lm3.8}
			\ H'(x_0) < 1 ,\ H'(x_1)=H'(x_2) = 1,
		\end{equation}
		
		Let
		\begin{equation*}
			F(x) = \frac{x}{1 - c_1 - x} .
		\end{equation*}
		Then we have
		\begin{equation*}
			F'(x) = \frac{1 - c_1}{x(1 - c_1 - x)} F(x),
		\end{equation*}
		and \noindent  \eqref{lm1.6} can be expressed as
		\begin{equation} \label{lm3.9}
			F(H(x)) = F(\widetilde{H}(x)) e^{(1 - c_1)(T - \tau_2)}.
		\end{equation}
		\noindent  Taking the derivative with respect to $x$ in \eqref{lm3.9},
		we get
		\begin{equation*}
			F'(H(x)) \cdot H'(x) = F'(\widetilde{H}(x)) \cdot \widetilde{H}'(x) \cdot e^{(1 - c_1)(T - \tau_2)}.
		\end{equation*}
		\noindent  which indicates that
		\begin{equation*}
			\frac{1 - c_1}{H(x)(1 - c_1 - H(x))} \cdot H'(x) = \frac{1 - c_1}{\widetilde{H}(x)(1 - c_1 - \widetilde{H}(x))} \cdot \widetilde{H}'(x).
		\end{equation*}
		\noindent  From \eqref{lm1.4}, we know that
		\begin{equation} \label{lm3.10}
			\widetilde{H}(x) = \frac{m_2 \overline{H}(x)}{(m_2 - 1) \overline{H}(x) + 1},
		\end{equation}
		and it follows from \eqref{lm1.2} that
		\begin{equation} \label{lm3.11}
			\widetilde{H}(x) = \frac{m_1 m_2 x}{m_1(m_2 - 1) x + 1}.
		\end{equation}
		\noindent  Taking the derivative with respect to $x$ in \eqref{lm1.4},
		we get
		\begin{equation} \label{lm3.12}
			\frac{\widetilde{H}'(x)}{(1 - \widetilde{H}(x))^2} = \frac{m_2 \overline{H}'(x)}{(1 - \overline{H}(x))^2},
		\end{equation}
		and it follows from  $\overline{H}'(x) = m_1$
		that
		\begin{equation} \label{lm3.13}
			\widetilde{H}'(x) = \frac{m_1 m_2}{(1 - m_1 x)^2} \cdot (1 - \widetilde{H}(x))^2.
		\end{equation}
		
		\noindent  Thus
		\begin{equation*}
			H'(x) = \frac{H(x)(1 - c_1 - H(x))}{x \cdot [(1 - c_1)(m_1( m_2 - 1) x + 1)- m_1 m_2 x]},
		\end{equation*}
		which implies that
		\begin{equation*}
			H'(x_0)=\frac{1 - c_1 - x_0}{(1 - c_1)(m_1( m_2 - 1) x_0 + 1)- m_1m_2x_0},
		\end{equation*}
		\begin{equation*}
			H'(x_1)=\frac{1 - c_1 - x_1}{(1 - c_1)(m_1( m_2 - 1) x_1 + 1)- m_1m_2x_1},
		\end{equation*}
		\begin{equation*}
			H'(x_2)=\frac{1 - c_1 - x_2}{(1 - c_1)(m_1( m_2 - 1) x_2 + 1)- m_1m_2x_2}.
		\end{equation*}
		
		Set
		\begin{equation*}
			P(x)= - (c_1 m_1 m_2 +(1-c_1)m_1-1)x.
		\end{equation*}
		\noindent If $(1 - c_1)(m_1( m_2 - 1) x_2 + 1)- m_1m_2x_2 < 0$, then conditions \eqref{lm3.1}, \eqref{lm3.2}, \eqref{lm3.3} and \eqref{lm3.4} can be equivalently written as
		\begin{equation*}
			P(x_0)=0 , \ P(x_2)=0 ,\ P(x_1)<0,
		\end{equation*}
		\begin{equation*}
			P(x_0)=0,\ P(x_2)<0,\ P(x_1)=0,
		\end{equation*}
		\begin{equation*}
			P(x_0)<0,\ P(x_2)>0,\ P(x_1)<0,
		\end{equation*}
		and
		\begin{equation*}
			P(x_0)<0,\ P(x_2)=0,\ P(x_1)=0,
		\end{equation*}
		respectively.
		\noindent Similarly, \eqref{lm3.5} and \eqref{lm3.6} can be equivalently written as
		\begin{equation*}
			P(x_0)=0,\ P(x_2)=0,\ P(x_1)<0,
		\end{equation*}
		and
		\begin{equation*}
			P(x_0)<0,\ P(x_2)<0,\ P(x_1)=0,
		\end{equation*}
		respectively. \eqref{lm3.7} and \eqref{lm3.8} can be equivalently written as
		\begin{equation*}
			P(x_0)=0,\ P(x_2)<0,\ P(x_1)<0\
		\end{equation*}
		and
		\begin{equation*}
			P(x_0)<0,\ P(x_2)=0,\ P(x_1)=0,
		\end{equation*}
		respectively. It is easy to see that  \eqref{lm3.1}-\eqref{lm3.8} cannot happen.
		If $(1 - c_1)(m_1( m_2 - 1) x_2 + 1)- m_1m_2x_2 > 0$, the sign of $P(x)$ is the opposite, we can similarly prove formulas \eqref{lm3.1}-\eqref{lm3.8} cannot happen.
		
		Now, we prove the  global stability of the $T$-periodic solution of system \eqref{main1.4}. We need only prove that the zero solution is locally unstable.
		Let $\{Q_{1t}\}_{t \geq 0}$ , $\{Q_{2t}\}_{t \geq 0}$ and $\{Q_{3t}\}_{t \geq 0}$ be the solution semiflows of the first, second and third equations of system \eqref{main1.4}, respectively. Let $DP(0)$ be the Jacobian derivative of $P$ at $0$. Denote by $f_1$, $f_2$ and $f_3$ the right-hand side vector fields of system \eqref{main1.4}, respectively. Then
		\begin{equation*}
			Df_1(0) = -d_1 , Df_2(0) = 1, Df_3(0) = 1-c_1 ,
		\end{equation*}
		Write $U_1(t, 0):=Q_{1t}(0)$, $ U_2(t, 0):=Q_{2t}(0)$ and $ U_3(t, 0):=Q_{3t}(0)$. Following the idea in \cite{Hsu.Zhao}, let $V_1(t, 0) = D U_1(t, 0)$, $V_2(t, 0) = DU_2(t, 0)$ and $V_3(t, 0) = DU_3(t, 0)$. Then the matrix functions $ V_1 := V_1(t, 0)$, $V_2 := V_2(t,0)$ and $V_3 := V_3(t, 0)$ satisfy
		\begin{equation*}
			\frac{dV_1(t)}{dt} = Df_1(U_1(t, 0))V_1(t), \quad V_1(0) = I,
		\end{equation*}
		\begin{equation*}
			\frac{dV_2(t)}{dt} = Df_2(U_2(t,0))V_2(t), \quad V_2(0) = I,
		\end{equation*}
		\begin{equation*}
			\frac{dV_3(t)}{dt} = Df_3(U_3(t,0))V_3(t), \quad V_3(0) = I.
		\end{equation*}			
		\noindent  It follows from the chain rule that
		\begin{equation*}
			DP(0) = DQ_{3(T-\tau_2)}(DQ_{2(\tau_2-\tau_1)}(Q_{1\tau_1}(0))) \cdot DQ_{2(\tau_2-\tau_1)}(Q_{1\tau_1}(0))\cdot DQ_{1\tau_1}(0).
		\end{equation*}
			\noindent  Then
		\begin{equation*}
			\begin{split}
				DQ_{1\tau_1}(0) &= V_1(\tau_1, 0), \\
				DQ_{2(\tau_2-\tau_1)}(Q_{1\tau_1(0)}) &= V_2(\tau_2-\tau_1, Q_{1\tau_1}(0)), \\
				DQ_{3(T-\tau_2)}(DQ_{2(\tau_2-\tau_1)}(Q_{1\tau_1}(0))) &= V_3(T-\tau_2, Q_{2(\tau_2-\tau_1)}, Q_{1\tau_1}(0))).
			\end{split}
		\end{equation*}

		Note that  $U_1(t, 0) = (u_1^*(t), 0)$ for $t \in (0, \tau_1]$.
		Thus
		\begin{equation*}
			Df_1(U_1(t, 0)) = -d_1,
		\end{equation*}			
		and consequently,
		\begin{equation*}
			V_1(\tau_1,0) = e^{\int_{0}^{\tau_1} -d_1 dt}.
		\end{equation*}
		Similarly, it follows from $U_2\left(t, Q_{1\tau_1}(0)\right) = \left(u_{2}^{*}(t), 0\right)$ for $t \in (\tau_1, \tau_2] $ that
		\begin{equation*}
			Df_2(U_2(t,  Q_{1\tau_1}(0))) = 1,
		\end{equation*}
		and hence
		\begin{equation*}
			V_2(\tau_2-\tau_1,Q_{1\tau_1}(0)) = e^{\int_{0}^{\tau_2-\tau_1}  dt} .
		\end{equation*}
		
		It follows from $U_3(t,Q_{2(\tau_2-\tau_1)}(Q_{1\tau_1}(0))) = (u_{3}^{*}(t), 0)$ for $t \in (\tau_2 , T] $ that
		\begin{equation*}
			Df_3(U_3(t, Q_{2(\tau_2-\tau_1)}(Q_{1\tau_1}(0)))) = 1-c_1,
		\end{equation*}
		and hence
		\begin{equation*}
			V_3(T-\tau_2, Q_{2(\tau_2-\tau_1)}, Q_{1\tau_1}(\omega_1))) = e^{\int_{0}^{T-\tau_2} (1-c_1) dt}.
		\end{equation*}
		Hence $DP(0)$ has  a positive eigenvalue
		\begin{equation*}
			r=e^{\int_0^{T-\tau_2}{(1 - c_1)}dt} \cdot e^{\int_0^{\tau_2-\tau_1}{}dt} \cdot e^{-d_1\tau_1} = e^{c_1(\tau_2-\tau_2^{*})} >1.
		\end{equation*}
		Apparently, the zero solution is unstable.
		
		$(\mathbf{ii})$
		When $c_1 = 1$ , it follows from \eqref{lm1.10} and $\tau_2 > (d_1+1)\tau_1$ that $\lim_{x \to 0} \frac{H(x)}{x} > 1$. There exists a sufficiently small $ \varepsilon > 0$, such that
		\begin{equation*}
			H(x) > x \text{ for } x \in (0, \varepsilon).
		\end{equation*}
		From Lemma \ref{lem.H<x}, we know that when $x \geq 1$ , $ H(x) < x$. Hence, there exists $x_0 \in [\varepsilon,1) $  such that $H(x_0) = x_0$, which indicates that $u(t; 0, x_0)$ is a $T$-periodic solution of system \eqref{main1.4}.
		
		Now we prove the uniqueness of the $T$-periodic solution of system \eqref{main1.4}.
		Assume that system \eqref{main1.4} has another $T$-periodic solution $x_1$, i.e. $H(x_1) = x_1$. Note that $ \lim_{x \to 0} \frac{H(x)}{x} > 1$ when $\tau_2 > \tau_2^*$, and $H(x) < x$ when $x \geq 1$.
		Without loss of generality, we assume $0 < x_0 < x_1 < 1$, where $x_0$ is the smallest fixed point of $H(x)$ and $x_1$ is the largest fixed point of $H(x)$ such that
		\begin{equation*}
			H(x) > x, \ x \in (0, x_0)\ \text{ and }\ H(x) < x , \ x \in (x_1, 1).
		\end{equation*}
		\noindent
		Then there must exist $x_2 \in [x_0, x_1]$ satisfying  conditions \eqref{lm3.1}-\eqref{lm3.8} .
		Let
		\begin{equation*}
			F(x) = e^{\frac{1}{x}},
		\end{equation*}
		Then
		\begin{equation*}
			F'(x) = -\frac{1}{x^2} F(x).
		\end{equation*}
		
		\noindent and \eqref{lm1.9} can be expressed as
		\begin{equation} \label{lm3.14}
			F(H(x)) = F(\widetilde{H}(x)) e^{T - \tau_2}.
		\end{equation}
		\noindent  Taking the derivative with respect to $x$ in \eqref{lm3.14},
		we get
		\begin{equation*}
			F'(H(x)) \cdot H'(x) = F'(\widetilde{H}(x)) \cdot \widetilde{H}'(x) \cdot e^{T - \tau_2},
		\end{equation*}
		\noindent  which indicates that,
		\begin{equation*}
			\frac{H'(x)}{H^2(x)}= \frac{\widetilde{H}'(x)}{\widetilde{H}^2(x)} .
		\end{equation*}
		\noindent  Thus we have
		\begin{equation*}
			H'(x) = \frac{H^2(x)}{m_1 m_2 x^2 }.
		\end{equation*}
		\noindent  We further obtain
		\begin{equation*}
			H'(x_0)=\frac{1}{m_1 m_2}, \ \ H'(x_1)=\frac{1}{m_1 m_2},\ \  H'(x_2)=\frac{1}{m_1 m_2}.
		\end{equation*}
		\noindent  It is easy to see that \eqref{lm3.1}-\eqref{lm3.8} cannot happen.
		
		Similarly to $(\mathbf{i})$, we can show that the zero solution is also unstable.
		From the results above, we know that the unique $T$-periodic solution for system \eqref{main1.4} is globally asymptotically stable.
	\end{proof}

	Similarly, we can obtain the following results for system \eqref{main1.5} and omit the proof.
	Set $\tau^{**}_1:=\frac{rT}{d_{2}+r}$ and  $\tau^{**}_2:=\frac{(d_2+r)\tau_1+(c_2-r)T}{c_2}$.

	\begin{theorem}\label{lem.v.solu}
		Assume that $\tau_1 < \tau^{**}_1$.   If $\tau_2 > \tau_2^{**}$, system \eqref{main1.5} has a unique positive $T$-periodic solution, which is globally asymptotically stable.
	\end{theorem}
	
	\begin{remark}\label{rem.(0,v)}
		From the results above, we know that system \eqref{main1.1}-\eqref{main1.3} has a unique semi-trivial $T$-periodic solution $(u^*(t), 0)$  if  $\tau_1 < \tau^{*}_1$ and  $\tau_2 > \tau_2^*$ and  $(0, v^*(t))$ if $\tau_1 < \tau^{**}_1$ and $\tau_2 > \tau_2 ^{**}$.\\
	\end{remark}

	\subsection{Extinction  conditions for single species}
 In this subsection, we continue to discuss the conditions of  extinction for single species $u$ and $v$.
	\begin{theorem} \label{lem.u.tao1}
		If $\tau_1 \geq \tau^{*}_1$, then all positive solutions of system \eqref{main1.4} tend to zero as $t \rightarrow +\infty$.
	\end{theorem}
	\begin{proof}
		We first study the following single-species model that incorporates seasonal succession
		\begin{equation}\label{eq1}
			\begin{cases}
				\displaystyle  \frac{du}{dt} = -d_1 u,  &\;\;  t \in (nT, nT + \tau_1],\\
				\displaystyle  \frac{du}{dt} = u (1 - u ),  &\;\;  t \in (nT + \tau_1, (n+1)T].
			\end{cases}
		\end{equation}	
		
		For any initial value $x>0$, it follow from the first and second equation of (\ref{eq1}) that $$u(\tau_1)=x\cdot e^{-d_1\tau_1}$$ and $$\frac{u(T)}{1-u(T)}=\frac{u(\tau_1)}{1-u(\tau_1)}\cdot e^{T-\tau_1}.$$
		Assume that equation \eqref{eq1} has a periodic solution, i.e., $u(T)=x$, which yields $(m_1 m_3-m_1)x=m_1 m_3-1$, where $m_1=e^{-d_1\tau_1}$ and $m_3:=e^{T-\tau_1}.$ To ensure $x>0$, we must have $m_1 m_3>1,$ i.e., $\tau_1 < \frac{T}{d_1+1}$. Hence, if $\tau_1 \geq \frac{T}{d_1+1}$, system \eqref{eq1} has no periodic solution, which implies that all positive solutions of system \eqref{eq1} tend to zero as $t \to \infty$.
		
		It should be noted that if the species $u$ go to extinction without grazing, there must have the same result for species $u$ in system \eqref{main1.4} when grazing is introduced.
	\end{proof}

	\begin{theorem}\label{lem.u.0}
		Assume that $\tau_1 < \tau^{*}_1$. If $\tau_2 \leq \tau_2^*$, then $H(x)<x$ for all $x>0$, i.e., all the positive solutions of system \eqref{main1.4} tend to zero as $t \rightarrow +\infty$.
	\end{theorem}
	\begin{proof}
		$(\mathbf{i})$
		When $c_1 \neq 1$, $\tau_2 = \tau_2^*$, it follows from Lemma \ref{lem.H<x} that $H(x) < x$ for any $x \geq 1$.
		For any $x \in (0,1),$ according to \eqref{lm1.2}, \eqref{lm1.4} and \eqref{lm1.6}, we see that
		\begin{equation*}
			\frac{H(x)}{x} = \frac{1-c_1-H(x)}{1-c_1-\widetilde{H}(x)} \frac{1-\widetilde{H}}{1-m_1 x}.
		\end{equation*}
		\noindent  Set
		\begin{equation*}
			G(x) = \frac{1-x}{1-c_1-x}.
		\end{equation*}
		Then
		\begin{equation*}
			G'(x) = \frac{c_1}{(1-c_1-x)^2}>0
		\end{equation*}
		and
		\begin{equation*}
			\frac{H(x)}{x} = \frac{1-H(x)}{1-m_1 x} \frac{G(\widetilde{H}(x))}{G(H(x))}.
		\end{equation*}
		
		If $ \widetilde{H}(x) < H(x) $, according to monotonicity of $ G(x) $, we get $ G(\widetilde{H}(x)) < G(H(x)). $
		So
		\begin{equation*}
			\frac{H(x)}{x} < \frac{1-H(x)}{1-m_1 x} < \frac{1-H(x)}{1-x}.
		\end{equation*}
		Therefore, $H(x) < x , x \in (0,1)$.
		
		If $ \widetilde{H}(x) > H(x) $, then $ 1-c_1-H(x) < 1-c_1-\widetilde{H}(x). $ When $ 1-c_1-\widetilde{H}(x) < 0, i.e., \widetilde{H}(x) > 1-c_1, $ we have
		\begin{equation*}
			\frac{1-c_1-H(x)}{1-c_1-\widetilde{H}(x)} <1.
		\end{equation*}
		So
		\begin{equation*}
			\frac{H(x)}{x} < \frac{1-\widetilde{H}(x)}{1-m_1 x} < \frac{1-H(x)}{1-m_1 x} < \frac{1-H(x)}{1-x}.
		\end{equation*}
		Therefore, $H(x) < x , x \in (0,1)$.
		
		When  $ 1-c_1-\widetilde{H}(x) > 0, i.e., \widetilde{H}(x) < 1-c_1, $ so $ 0 < c_1 < 1.$ We can get $ m_1 m_2 < 1. $ So from \eqref{lm3.11}
		\begin{equation*}
			\frac{H(x)}{x} < \frac{\widetilde{H}(x)}{x} = \frac{m_1 m_2}{m_1(m_2 -1) x+1} < m_1 m_2 <1.
		\end{equation*}
		Therefore, $H(x) < x , x \in (0,1)$.
		
		$(\mathbf{ii})$
		When $c_1 \neq 1$, $\tau_2 < \tau_2^*$, from \eqref{lm1.7}, we know that $\lim_{x \to 0} \frac{H(x)}{x} < 1$. There exists a small $\varepsilon > 0$ such that $H(x) < x $ for $ x \in (0, \varepsilon]$.
		Also, from Lemma \ref{lem.H<x}, we know that when $x \geq 1$ , $ H(x) < x$. Assume there exists $\widetilde{x} \in (\varepsilon, 1) $ , such that $ H(\widetilde{x}) \geq \widetilde{x} $. Then there must exist $\hat{x} \in (\epsilon,\widetilde{x}]$ , such that $H(\hat{x})=\hat{x}$ ,$ H'(\hat{x}) \geq 1.$
		\noindent  While
		\begin{equation*}
			P(\hat{x})=-(c_1 m_1 m_2 +(1-c_1) m_2 -1)\hat{x}.
		\end{equation*}
		\noindent  If  $c_1 m_1 m_2 +(1-c_1) m_2 -1 > 0$, then $P(\hat{x})<0$. We get $H'(\hat{x})<1$, a contradiction.
		\noindent  If  $c_1 m_1 m_2 +(1-c_1) m_2 -1 < 0$, then $P(\hat{x})>0$. We get $H'(\hat{x})>1$ and we can conclude there exist $\overline{x} \in (\widetilde{x}, 1)$ ,  such that $H(\overline{x})=\overline{x}$ , $ H'(\overline{x})\leq 1$ , $ P(\overline{x})\leq 0$ , a contradiction.
		
		$(\mathbf{iii})$
		When $c_1=1$, $\tau_2 \leq \tau_2^*=(d_1+1)\tau_1$, we can get $m_1 m_2 \leq 1$.  From Lemma \ref{lem.H<x}, we know that when $x \geq 1 , H(x) < x $.			
		\noindent  For any $x \in (0,1),$ according to \eqref{lm1.2}, \eqref{lm1.4} and \eqref{lm1.8}, we see that
		\begin{equation*}
			H(x)  = \frac{\widetilde{H}(x)}{(T-\tau_2)\widetilde{H}(x)+1}
			= \frac{m_1 m_2 x}{m_1 m_2 (T-\tau_2)x+m_1(m_2-1)x+1}.
		\end{equation*}
		\noindent  Hence
		\begin{equation*}
			\frac{H(x)}{x}=\frac{1}{(T-\tau_2+1-e^{\tau_1-\tau_2})x+\frac{1}{m_1 m_2}}<1,
		\end{equation*}
		i.e., $H(x)<x.$ It follows from Lemma \ref{lemma2.1} that all positive solutions of system \eqref{main1.4} tend to zero as $t \rightarrow +\infty$. The proof is completed.
	\end{proof}
Similarly, we can obtain the following results for system \eqref{main1.5} and omit the proof..
	\begin{theorem}\label{lem.v.tao2}
		If $\tau_1 \geq \tau^{**}_1$, then all positive solutions of system \eqref{main1.5} tend to zero as $t \rightarrow +\infty $.
	\end{theorem}
	
	\begin{theorem}\label{lem.v0}
		Assume that $\tau_1 < \tau^{**}_1$.
		If $\tau_2 \leq \tau_2^{**}$, then $ K(y)<y $ for any $y>0,$ i.e., all the positive solutions of system \eqref{main1.5} tend to zero as $t \rightarrow +\infty$.
	\end{theorem}
	\begin{remark}
		Theorems \ref{lem.u.tao1} and \ref{lem.u.0} indicate that species $u$ tends to extinction if the duration of  dry season is excessively large ($\tau_1 > \tau^{*}_1$) or  the growth period is too short ($\tau_2 \leq \tau_2^{*}$) even if the duration of  dry season is small ($\tau_1 < \tau^{*}_1$).
		Similar results hold for species $v$ (see Theorems \ref{lem.v.tao2} and \ref{lem.v0}).
	\end{remark}

	\section{Competitive outcomes in two-species vegetation systems under seasonal grazing}
	 \ \ \ The single-species analysis above identifies thresholds for $\tau_{1}$ (dry season) and $\tau_{2}$ (grazing-related growth season) that determine vegetation persistence and extinction. In this section, we extend the single-species model to a two-species competition system to explore how grazing duration and intensity regulate competitive outcomes (e.g., exclusion, coexistence, or system collapse).
	\subsection{Competitive exclusion and global stability of semi-trivial solutions}	 \ \ \ \   Denote by $W(t; 0, \omega)$ the unique nonnegative global solution of system \eqref{main1.1}-\eqref{main1.3} on $[0, \infty)$ with the initial condition $\omega = (x, y)$. Let $\{Q_{1t}\}_{t \geq 0}$ , $\{Q_{2t}\}_{t \geq 0}$ and $\{Q_{3t}\}_{t \geq 0}$ be the solution semiflows associated with
	
	\begin{equation}\label{eq4.1}
		\begin{cases}
			\displaystyle  \frac{du}{dt} = -d_1 u,  \\
			\displaystyle  \frac{dv}{dt} = -d_2 v,  \\
			\displaystyle (u(0), v(0)) \in \mathbb{R}_+^2,
		\end{cases}
	\end{equation}
	\begin{equation}\label{eq4.2}
		\begin{cases}
			\displaystyle\frac{du}{dt} = u(1 - u - b_1 v) , \\
			\displaystyle\frac{dv}{dt} = r v(1 - v - b_2 u) , \\
			\displaystyle(u(0), v(0)) \in \mathbb{R}_+^2
		\end{cases}
	\end{equation}
	and
	\begin{equation}\label{eq4.3}
		\begin{cases}
			\displaystyle\frac{du}{dt} = u(1 - u - b_1 v) - c_1 u, \\
			\displaystyle\frac{dv}{dt} = r v(1 - v - b_2 u) - c_2 v, \\
			\displaystyle(u(0), v(0)) \in \mathbb{R}_+^2,
		\end{cases}
	\end{equation}
	respectively. Since system \eqref{main1.1}-\eqref{main1.3} is $T$-periodic, we consider its associated period map $P$ defined as $P(\omega) = W(T; 0, \omega)$ for $\omega \in \mathbb{R}^2$. Thus $P(\omega) = Q_{3(T-\tau_2)}(Q_{2(\tau_2-\tau_1)}(Q_{1\tau_1}(\omega)))$ for all $ \omega \in \mathbb{R}^2$, that is, $P = Q_{3(T-\tau_2)}(Q_{2(\tau_2-\tau_1)}(Q_{1\tau_1}))$.
	
	Under certain conditions, it follows from Lemma \ref{lem.u.solu} that there exists a unique $x_0 > 0$ such that $u^*(t) = u(t; 0, x_0)$ is the unique $T$-periodic solution of system \eqref{main1.4} with
	\begin{equation*}
		u^*(t) =
		\begin{cases}
			u_1^*(t), & t \in (0, \tau_1], \\
			u_2^*(t -  \tau_1), & t \in (\tau_1, \tau_2], \\
			u_3^*(t -  \tau_2), & t \in (\tau_2, T], \\
		\end{cases}
	\end{equation*}
	where $u_1^*(t)$ is the solution to the first equation of system \eqref{main1.4} with the initial value $u_1^*(0) = x_0$, $u_2^*(t)$ is the solution to the second equation of system \eqref{main1.4} with the initial value $u_2^*(0) = u_1^*(\tau_1) = \bar{H}(x_0)$, and $u_3^*(t)$ is the solution to the third equation of system \eqref{main1.4} with the initial value $u_3^*(0) = u_2^*(\tau_2-\tau_1) = \widetilde{H}(x_0)$.
	
	Similarly, there exists $y_0 > 0$ such that $v^*(t) = v(t; 0, y_0)$ is the unique $T$-periodic solution of system \eqref{main1.5} with
	\begin{equation*}
		v^*(t) =
		\begin{cases}
			v_1^*(t), & t \in (0, \tau_1], \\
			v_2^*(t -  \tau_1), & t \in (\tau_1, \tau_2], \\
			v_3^*(t -  \tau_2), & t \in (\tau_2, T],
		\end{cases}
	\end{equation*}
	where $v_1^*(t)$ is the solution to the first equation of system \eqref{main1.5} with the initial value $v_1^*(0) = y_0$, and $v_2^*(t) $ is the solution to the second equation of system \eqref{main1.5} with the initial value $v_2^*(0) = v_1^*(\tau_1) = \bar{K}(y_0)$, and $v_3^*(t)$ is the solution to the third equation of system \eqref{main1.5} with the initial value $v_3^*(0) = v_2^*(\tau_2-\tau_1) = \widetilde{K}(y_0)$.
	
	Write $\omega_1 = (x_0, 0)$ and $\omega_2 = (0, y_0)$. Under certain conditions, it follows from   that $W(t; 0, \omega_1)$ and $W(t; 0, \omega_2)$ are two semi-trivial $T$-periodic solutions of \eqref{main1.1}-\eqref{main1.3}.
	Let $DP(\omega_1)$ and $DP(\omega_2)$ be the Jacobian matrices of $P$ at $\omega_1$ and $\omega_2$ with their spectral radius $\rho(DP(\omega_1))$ and $ \rho(DP(\omega_2))$, respectively.

Then we have the following results concerning the local stability of the semitrivial fixed point $\omega_1$ and $\omega_2$ of the period map $P$.
	
	\begin{theorem}\label{lemma4.1}
		$(\mathbf{i})$Assume that  $\tau_1 < \tau^{*}_1$ and  $\tau_2 > \tau_2^*$. Then the semi-trivial $T$-periodic solution $(u^*(t), 0)$ of system \eqref{main1.1}-\eqref{main1.3}  exist and it is stable if and only if $\frac{\tau_2-\tau_2^{**}}{\tau_2-\tau_2^*}<\frac{rb_2c_1}{c_2}$.
		
		\((\mathbf{ii})\)Assume that $\tau_1 < \tau^{**}_1$ and $\tau_2 > \tau_2^{**}$. Then the semi-trivial T-periodic solution $(0,v^*(t))$ of system \eqref{main1.1}-\eqref{main1.3}  exist and it  is stable if and only if $\frac{\tau_2-\tau_2^{*}}{\tau_2-\tau_2^{**}}<\frac{b_1c_2}{rc_1}$.
	\end{theorem}		
	\begin{proof}
		$(\mathbf{i})$ Assume that $\tau_1 <  \tau^{*}_1$ and  $\tau_2 > \tau_2^*$. Then from Remark \ref{rem.(0,v)}, we know that system \eqref{main1.1}-\eqref{main1.3} has a unique semi-trivial  $T$-periodic solution $(u^*(t), 0)$.
		
		Denote by $f_1$, $f_2$ and $f_3$ the right-hand side vector fields of system \eqref{main1.1}-\eqref{main1.3}, respectively. Then
		\begin{equation*}
			Df_1(\omega) = \begin{bmatrix} -d_1 & 0 \\ 0 & -d_2 \end{bmatrix},
		\end{equation*}
		\begin{equation*}
			Df_2(\omega) = \begin{bmatrix} 1 - 2x - b_1y  & -b_1x \\ -rd_2y & r(1-2y-b_2x) \end{bmatrix}
		\end{equation*}
		and
		\begin{equation*}
			Df_3(\omega) = \begin{bmatrix} 1 - 2x - b_1y - c_1  & -b_1x \\ -rd_2y & r(1-2y-b_2x) - c_2  \end{bmatrix}
		\end{equation*}
		for $\omega = (x,y) \neq (0,0)$.
		
		Write $U_1(t, \omega):=Q_{1t}(\omega)$, $ U_2(t, \omega):=Q_{2t}(\omega)$ and $ U_3(t, \omega):=Q_{3t}(\omega)$. Following the idea in \cite{Hsu.Zhao}, let $V_1(t, \omega) = D_\omega U_1(t, \omega)$, $V_2(t, \omega) = D_\omega U_2(t, \omega)$ and $V_3(t, \omega) = D_\omega U_3(t, \omega)$. Then the matrix functions $ V_1 := V_1(t, \omega)$, $V_2 := V_2(t,\omega)$ and $V_3 := V_3(t, \omega)$ satisfy
		\begin{equation*}
			\frac{dV_1(t)}{dt} = Df_1(U_1(t, \omega))V_1(t), \quad V_1(0) = I,
		\end{equation*}
		\begin{equation*}
			\frac{dV_2(t)}{dt} = Df_2(U_2(t,\omega))V_2(t), \quad V_2(0) = I,
		\end{equation*}
		\begin{equation*}
			\frac{dV_3(t)}{dt} = Df_3(U_3(t,\omega))V_3(t), \quad V_3(0) = I.
		\end{equation*}			
		\noindent  It follows from the chain rule that
		\begin{equation*}
			DP(\omega_1) = DQ_{3(T-\tau_2)}(DQ_{2(\tau_2-\tau_1)}(Q_{1\tau_1}(\omega_1))) \cdot DQ_{2(\tau_2-\tau_1)}(Q_{1\tau_1}(\omega_1))\cdot DQ_{1\tau_1}(\omega_1).
		\end{equation*}
		
		\noindent  Then
		\begin{equation*}
			\begin{split}
				DQ_{1\tau_1}(\omega_1) &= V_1(\tau_1, \omega_1), \\
				DQ_{2(\tau_2-\tau_1)}(Q_{1\tau_1(\omega_1)}) &= V_2(\tau_2-\tau_1, Q_{1\tau_1}(\omega_1)), \\
				DQ_{3(T-\tau_2)}(DQ_{2(\tau_2-\tau_1)}(Q_{1\tau_1}(\omega_1))) &= V_3(T-\tau_2, Q_{2(\tau_2-\tau_1)}, Q_{1\tau_1}(\omega_1))).
			\end{split}
		\end{equation*}

		Note that  $U_1(t, \omega_1) = (u_1^*(t), 0)$ for $t \in (0, \tau_1]$.
		Thus
		\begin{equation*}
			Df_1(U_1(t, \omega_1)) = \begin{bmatrix} -d_1 & 0 \\ 0 & -d_2 \end{bmatrix},
		\end{equation*}			
		and consequently,
		\begin{equation*}
			V_1(\tau_1,\omega_1) = \begin{bmatrix} e^{-d_1 \tau_1} & 0 \\ 0 & e^{-d_2\tau_1} \end{bmatrix}.
		\end{equation*}
		
		Similarly, it follows from $U_2\left(t, Q_{1\tau_1}(\omega_1)\right) = \left(u_{2}^{*}(t), 0\right)$ for $t \in (\tau_1, \tau_2] $ that
		\begin{equation*}
			Df_2(U_2(t,  Q_{1\tau_1}(\omega_1))) = \begin{bmatrix} 1 - 2 u_2^*(t) & -b_1u_2^*(t) \\ 0 & r(1-b_2u_2^*(t)) \end{bmatrix},
		\end{equation*}
		and hence
		\begin{equation*}
			V_2(\tau_2-\tau_1,Q_{1\tau_1}(\omega_1)) = \begin{bmatrix} e^{\int_0^{\tau_2-\tau_1}{(1 - 2 u_2^*(t))}} dt & * \\ 0 & e^{\int_0^{\tau_2-\tau_1}r(1-b_2u_2^*(t))dt}  \end{bmatrix}.
		\end{equation*}
		
		It follows from $U_3(t,Q_{2(\tau_2-\tau_1)}(Q_{1\tau_1}(\omega_1))) = (u_{3}^{*}(t), 0)$ for $t \in (\tau_2 , T] $ that
		\begin{equation*}
			Df_3(U_3(t, Q_{2(\tau_2-\tau_1)}(Q_{1\tau_1}(\omega_1)))) = \begin{bmatrix} 1 - 2 u_3^*(t) - c_1 & -b_1u_3^*(t) \\ 0 & r(1-b_2u_3^*(t)) - c_2 \end{bmatrix},
		\end{equation*}
		and hence
		\begin{equation*}
			V_3(T-\tau_2, Q_{2(\tau_2-\tau_1)}, Q_{1\tau_1}(\omega_1))) = \begin{bmatrix} e^{\int_0^{T-\tau_3}{(1 - 2 u_2^*(t)-c_1)}} dt & * \\ 0 & e^{\int_0^{T-\tau_2}r(1-b_2u_3^*(t))-c_2dt} \end{bmatrix}.
		\end{equation*}
		
		Therefore matrix $DP(\omega_1)$ has two positive eigenvalues
		\begin{equation*}
			\lambda_1 = e^{\int_0^{T-\tau_2}{(1 - 2 u_3^*(t)-c_1)}dt} \cdot e^{\int_0^{\tau_2-\tau_1}{(1 - 2 u_2^*(t))}dt} \cdot e^{-d_1\tau_1},
		\end{equation*}
		\begin{equation*}
			\lambda_2 = e^{\int_0^{T-\tau_2}r(1-b_2u_3^*(t))-c_2dt} \cdot e^{\int_0^{\tau_2-\tau_1}r(1-b_2u_2^*(t))dt} \cdot e^{-d_2\tau_1}.
		\end{equation*}			
		It follows from the first equation in \eqref{main1.4} that
		\begin{equation*}
			\frac{du^{*}_{1}(t)}{dt}\frac{1}{u^{*}_{1}(t)} = -d_1.
		\end{equation*}			
		Integrating the above equation in $t$ from $0$ to $\tau_1$, we  obtain
		\begin{equation}\label{eq4.4}
			\ln u^{*}_{1}(\tau_1) - \ln x_{0} = \int_{0}^{\tau_1} -d_1 dt.
		\end{equation}
		
		Similarly, from the second and third  equation in \eqref{main1.4}, we have
		
		\begin{equation}\label{eq4.5}
			\ln u_2^*(\tau_2-\tau_1) - \ln u_1^*(\tau_1) = \int_0^{\tau_2 - \tau_1} (1 - u_2^*(t))  dt.
		\end{equation}
		
		and
		\begin{equation}\label{eq4.6}
			\ln x_0 - \ln u_2^*(\tau_2-\tau_1) = \int_0^{T - \tau_2} (1 - u_3^*(t) - c_1)  dt.
		\end{equation}
		
		By \eqref{eq4.4}-\eqref{eq4.6} and a straightforward computation, we have
		\begin{equation}\label{eq4.7}
			(1-c_1)T+c_1\tau_2-(1+d_1)\tau_1 = \int_0^{\tau_2-\tau_1}u_2^*(t) dt +  \int_0^{T - \tau_2}u_3^*(t) dt.
		\end{equation}
		Hence
		\begin{eqnarray*}
			\lambda_1&=&e^{\int_0^{T-\tau_2}{(1 - 2 u_3^*(t)-c_1)}dt} \cdot e^{\int_0^{\tau_2-\tau_1}{(1 - 2 u_2^*(t))}dt} \cdot e^{-d_1\tau_1} \\
			&=&e^{(c_1-1)T-c_1\tau_2+(1+d_1)\tau_1} \\
			&=&e^{-c_1(\tau_2-\tau_2^*)} \\
			&<&1,
		\end{eqnarray*}
		\begin{eqnarray*}
			\lambda_2&=&e^{\int_0^{T-\tau_2}[r(1-b_2u_3^*(t))-c_2]dt} \cdot e^{\int_0^{\tau_2-\tau_1}r(1-b_2u_2^*(t))dt} \cdot e^{-d_2\tau_1} \\
			&=&e^{[(r-c_2)T+c_2\tau_2-(r+d_2)\tau_1)]-rb_2[(1-c_1)T+c_1\tau_2-(1+d_1)\tau_1]}  \\
			&=&e^{c_{2}(\tau_2-\tau_2^{**})-rb_{2}c_{1}(\tau_2-\tau_2^{*})}.
		\end{eqnarray*}
		We have	$\lambda_2<1$ if and only if $\frac{\tau_2-\tau_2^{**}}{\tau_2-\tau_2^*}<\frac{rb_2c_1}{c_2}$, which implies that $(u^*(t), 0)$ is stable if and only if $\frac{\tau_2-\tau_2^{**}}{\tau_2-\tau_2^*}<\frac{rb_2c_1}{c_2}$.
		
		$(\mathbf{ii})$ From Remark \ref{rem.(0,v)}, we see that system \eqref{main1.1}-\eqref{main1.3} has a unique semi-trivial $T$-periodic solution $(0,v^*(t))$. By similar analysis as in the proof of part ($\mathbf{i}$), we have matrix $DP(\omega_2)$ has two positive eigenvalues
		\begin{eqnarray*}
			\lambda_3&=&e^{\int_0^{T-\tau_2}(r(1-2v_3^*(t))-c_2)dt} \cdot e^{\int_0^{\tau_2-\tau_1}r(1-2u_2^*(t))dt} \cdot e^{-d_2\tau_1} \\
			&=&e^{(c_2-r)T-c_2\tau_2+(r+d_2)\tau_1} \\
			&=&e^{-c_2(\tau_2-\tau_2^{**})} \\
			&<&1.
		\end{eqnarray*}
		
		\begin{eqnarray*}
			\lambda_4&=&e^{\int_0^{T-\tau_2}{(1 - b_1 v_3^*(t)-c_1)}dt} \cdot e^{\int_0^{\tau_2-\tau_1}{(1 - b_1 v_2^*(t))}dt} \cdot e^{-d_1\tau_1} \\
			&=&e^{[(1-c_1)T+c_1\tau_2-(1+d_1)\tau_1]-\frac{b_1}{r}[(r-c_2)T+c_2\tau_2-(r+d_2)\tau_1)]} \\
			&=&e^{c_{1}(\tau_2-\tau_2^{*})-\frac{b_{1}c_{2}}{r}(\tau_2-\tau_2^{**})}.
		\end{eqnarray*}
		We know that 	$\lambda_4<1$ if and only if $\frac{\tau_2-\tau_2^{*}}{\tau_2-\tau_2^{**}}<\frac{b_1c_2}{rc_1}$, which indicates that	$(0,v^*(t))$ is stable if and only if $\frac{\tau_2-\tau_2^{*}}{\tau_2-\tau_2^{**}}<\frac{b_1c_2}{rc_1}$. The proof  is complete.
		\end{proof}

		\begin{theorem}\label{theorem4.2}
		Assume that $\tau_1 < \tau^{*}_1$, $\tau_2^*<\tau_2 <\tau_2^{**}$. Then the unique semi-trivial $T$-periodic solution $(u^*(t),0)$ of system \eqref{main1.1}-\eqref{main1.3} is globally asymptotically stable in $Int (\mathbb{R}_{+}^{2})$.
		\end{theorem}
		\begin{proof}
		From Theorem \ref{lemma4.1}, we see  that the $T$-periodic solution $\left(u^{*}(t), 0\right)$ originating from $\left(x_{0}, 0\right)$ is locally stable.
		It remains to prove that every solution of system \eqref{main1.1}-\eqref{main1.3} goes to $\left(u^{*}(t), 0\right)$.
		
		Note that $\tau_2^*<\tau_2 <\tau_2^{**}$. Similar to the proof of Theorem \ref{lemma4.1}, we can show that matrix $ DP((0,0))$ has two positive eigenvalues
		\begin{eqnarray*}
			\lambda_5 = e^{(1-c_1)(T-\tau_2)+(\tau_2-\tau_1)-d_1\tau_1}=e^{c_1(\tau_2-\tau_2^*)}>1,
		\end{eqnarray*}
		\begin{eqnarray*}
			\lambda_6 = e^{(r-c_2)(T-\tau_2)+r(\tau_2-\tau_1)-d_2\tau_1} = e^{c_2(\tau_2-\tau_2^{**})}<1.
		\end{eqnarray*}
		Hence, the zero solution $E_{0}$ of system \eqref{main1.1}-\eqref{main1.3} is unstable. It follows from $\tau_2 <\tau_2^{**}$ and Theorem \ref{lem.v.solu} that system \eqref{main1.1}-\eqref{main1.3} has no semi-trivial $T $-periodic solution $\left(0, v^*(t)\right)$.
		
		We are ready to show that  system \eqref{main1.1}-\eqref{main1.3} has no positive $T$-periodic solution.
		Assume that system \eqref{main1.1}-\eqref{main1.3} has a positive $T$-periodic solution $(\bar{u}(t),\bar{v}(t))$ originating from $(\bar{x},\bar{y})$, with
		\begin{equation}\label{ubar1}
			\bar{u}(t) =
			\begin{cases}
				\bar{u}_{1}(t), \, & t \in (0, \tau_1], \\
				\bar{u}_{2}(t-\tau_1), & t \in (\tau_1, \tau_2], \\
				\bar{u}_{3}(t-\tau_2), & t \in (\tau_2, T]
			\end{cases}
		\end{equation}
		and			
		\begin{equation}\label{ubar2}
			\bar{v}(t) =
			\begin{cases}
				\bar{v}_{1}(t), & t \in (0, \tau_1], \\
				\bar{v}_{2}(t-\tau_1), & t \in (\tau_1, \tau_2], \\
				\bar{v}_{3}(t-\tau_2), & t \in (\tau_2, T].
			\end{cases}
		\end{equation}
		It follows from equation \eqref{main1.1}, \eqref{main1.2} and \eqref{main1.3} that
		\begin{equation}\label{p1}
			\begin{cases}
				\ln\bar{u}_{1}(\tau_1)-\ln\bar{x}=\int_{0}^{\tau_1}-d_{1} d t,\\
				\ln\bar{v}_{1}(\tau_1)-\ln\bar{y}=\int_{0}^{\tau_1}-d_{2} d t,
			\end{cases}
		\end{equation}
		\begin{equation}\label{p2}
			\begin{cases}
				\ln\bar{u}_{2}(\tau_2-\tau_1)-\ln\bar{u}_{1}(\tau_1)=\int_{0}^{\tau_2-\tau_1}(1-\bar{u}_{2}(t)-b_{1}\bar{v}_{2}(t)) d t,\\
				\ln\bar{v}_{2}(\tau_2-\tau_1)-\ln\bar{v}_{1}(\tau_1)=\int_{0}^{\tau_2-\tau_1}r(1-\bar{v}_{2}(t)-b_{2}\bar{u}_{2}(t)) d t,
			\end{cases}
		\end{equation}		
		and
		\begin{equation}\label{p3}
			\begin{cases}
				\ln\bar{x}-\ln\bar{u}_{2}(\tau_2-\tau_1)=\int_{0}^{T-\tau_2}(1-\bar{u}_{3}(t)-b_{1}\bar{v}_{3}(t)-c_1) dt,\\
				\ln\bar{y}-\ln\bar{v}_{2}(\tau_2-\tau_1)=\int_{0}^{T-\tau_2}(r-r\bar{v}_{3}(t)-rb_{2}\bar{u}_{3}(t)-c_2) dt.
			\end{cases}
		\end{equation}			
		Adding the second equation of \eqref{p1}, \eqref{p2} and \eqref{p3}, we obtain
		\begin{equation*}
			\begin{aligned}
				&\int_0^{\tau_2-\tau_1}(r\bar{v}_2(t)+rb_2\bar{u}_2(t))dt+\int_0^{T-\tau_2}(r\bar{v}_3(t)+rb_2\bar{u}_3(t))dt \\
				&=(r-c_2)(T-\tau_2)+r(\tau_2-\tau_1)-d_2\tau_1 \\
				&=c_2(\tau_2-\tau_2^{**})<0,
			\end{aligned}
		\end{equation*}
		which is a contradiction. Hence, system \eqref{main1.1}-\eqref{main1.3} has no positive $T$-periodic solution.
		From lemma 2.2($\mathbf{iii}$) of \cite{Hsu.Zhao}, we know that every solution of \eqref{main1.1}-\eqref{main1.3} goes to $(u^*(t),0)$ in $Int (\mathbb{R}_{+}^{2})$. The proof is complete.
		\end{proof}
		\begin{remark}From a biological perspective, Theorem \ref{theorem4.2} indicates that in semi-arid ecosystems, when the grazing period for vegetation $u$ is relatively short and that for vegetation $v$ is long, vegetation $u$ can better adapt to the environmental conditions, obtain more resources for growth and reproduction, and thus outcompete vegetation $v$.\end{remark}

		Similar to the proof of Theorem \ref{theorem4.2}, we can obtain the following result for semi-trivial $T$-periodic solution $(0,v^*(t))$ of system \eqref{main1.1}-\eqref{main1.3}.
		\begin{theorem}\label{theorem4.3}
		Assume that $\tau_1 < \tau^{**}_1$, $\tau_2^{**}<\tau_2 < \tau_2^{*}$.  Then the unique semi-trivial $T$-periodic solution $(0,v^*(t))$ of system \eqref{main1.1}-\eqref{main1.3} is globally asymptotically stable in $Int (\mathbb{R}_{+}^{2})$.
		\end{theorem}
		\begin{proof}
		Since	$\tau_1 < \tau^{**}_1$, $\tau_2^{**}<\tau_2 < \tau_2^{*}$,	from Theorem \ref{lemma4.1}, we see that the $T$-periodic solution $(0,v^*(t))$ is locally stable and the trivial steady state $E_{0}$ of system \eqref{main1.1}-\eqref{main1.3} is unstable. Sine $\tau_2 < \tau_2^*$, by Theorem \ref{lem.u.0}, we know that system \eqref{main1.1}-\eqref{main1.3} has no semi-trivial $T$-periodic solution $(u^*(t),0)$.
		
		We next prove system \eqref{main1.1}-\eqref{main1.3} has no positive $T$-periodic solution by comparison theorem. Similar to the proof of Theorem \ref{theorem4.2}, we assume that system \eqref{main1.1}-\eqref{main1.3} has a positive $T$-periodic solution $(\bar{u}(t),\bar{v}(t))$ originating from $(\bar{x},\bar{y})$. We add the first equations of \eqref{p1}, \eqref{p2} and \eqref{p3}, and obtain
		\begin{equation*}
			\begin{aligned}
				&\int_0^{\tau_2-\tau_1}(\bar{u}_2(t)+b_1\bar{v}_2(t))dt+\int_0^{T-\tau_2}(\bar{u}_3(t)+b_1\bar{v}_3(t))dt \\
				&=-d_1\tau_1+\tau_2-\tau_1+(1-c_1)(T-\tau_2) \\
				&=c_1(\tau_2-\tau_2^{*})<0,
			\end{aligned}
		\end{equation*}
		which is a contradiction. Hence, system \eqref{main1.1}-\eqref{main1.3} has no positive $T$-periodic solution.
		From Lemma 2.2($\mathbf{iii}$) of \cite{Hsu.Zhao}, the unique semi-trivial $T$-periodic solution $(0,v^*(t))$ of system \eqref{main1.1}-\eqref{main1.3} is globally attractive. The proof is complete.
		\end{proof}

		\subsection{System collapse: extinction of both vegetation species}	
		We now state the result about the stability of the trivial steady state $E_0$.
		
		\begin{theorem}\label{theorem4.1}
		Assume that one of the following conditions is satisfied:\\
		$(\mathbf{i})$ $\tau_1 \geq \text{max} \{\tau^{*}_1,\tau^{**}_1\}$.\\
		$(\mathbf{ii})$   $\tau^{*}_1\leq\tau_1 < \tau^{**}_1$, $\tau_2\leq\tau_2^{**}$.\\
		$(\mathbf{iii})$ $\tau^{**}_1\leq\tau_1 < \tau^{*}_1$ , $\tau_2\leq\tau_2^*$.\\
		$(\mathbf{iv})$ $\tau_1 < \text{min} \{\tau^{*}_1,\tau^{**}_1\}$, $\tau_2\leq \text{min} \{\tau_2^*,\tau_2^{**}\}$.\\
		Then the trivial steady state $E_0$ of system \eqref{main1.1}-\eqref{main1.3} is globally asymptotically stable  in $Int (\mathbb{R}^{2}_{+})$.
		\begin{proof}
		When conditions (i)	and (iv) are satisfied, one can easily show that the trivial steady state $E_0$ of system \eqref{main1.1}-\eqref{main1.3} is globally asymptotically stable. We only prove case (ii). Denote by $W(t; 0, \omega)=(u(t),v(t))$ the solution of system \eqref{main1.1}-\eqref{main1.3} with the initial value $ (u(0),v(0)) = \omega = (x, y) \in \mathbb{R}^2_+ $. Then $u(t)\geq 0$ and $v(t)\geq 0$ for $t\geq 0$.
			
			Let $\hat{u}(t)$ be the solution of system \eqref{main1.4} with initial value $\hat{u}(0) = x$. Then
			\begin{equation*}
				u(t) \leq \hat{u}(t), \quad t \geq 0.
			\end{equation*}
			Note that one of the conditions $(\mathbf{i})$-$(\mathbf{iv})$ is satisfied. It follows from Theorem \ref{lem.u.tao1} and Theorem \ref{lem.u.0} that the zero solution of system \eqref{main1.4} is globally asymptotically stable. Therefore
			\begin{equation*}
				\lim_{t \to \infty} u(t) = 0.
			\end{equation*}			
			This, together with the theory of asymptotically autonomous systems, implies that the long time behavior of \eqref{main1.1}-\eqref{main1.3} is determined by its limiting system, i.e., system \eqref{main1.5}. If  $\tau_1 < \tau_1^{**}$, $\tau_2\leq\tau_2^{**}$ holds, then by Theorems \ref{lem.v.tao2} and \ref{lem.v0}, the trivial steady state $E_0$ of system \eqref{main1.5} is globally asymptotically stable. The proof is completed.
		\end{proof}
		\end{theorem}
		
		\subsection{Coexistence  of two competition vegetation species}	
		
		\begin{theorem}\label{theorem4.4}
		Assume that the following conditions  are satisfied:\\
		$(\mathbf{i})$  $\tau_1 < \text{min} \{\tau^{*}_1,\tau^{**}_1\}$.\\
		$(\mathbf{ii})$   $\tau_2> \text{max} \{\tau_2^*,\tau_2^{**}\}$.\\
		$(\mathbf{iii})$  $\frac{r b_2 c_1}{c_2} < \frac{\tau_2 - \tau_2^{**}}{\tau_2 - \tau_2^*} < \frac{r c_1}{b_1 c_2}. $\\
		Then system \eqref{main1.1}-\eqref{main1.3} has a unique globally asymptotically stable positive $T$-periodic solution $(\bar{u}(t),\bar{v}(t))$ originate from $(\bar{x},\bar{y})$ in $Int (\mathbb{R}^{2}_{+})$.
		\end{theorem}
		\begin{proof}
		It follows
		from conditions $(\mathbf{i})$, $(\mathbf{ii})$ and $(\mathbf{iii})$ that the trivial steady state  $E_{0}$, two semi-trivial $T$-periodic solutions $(u^*(t),0)$ and $(0,v^*(t))$ are unstable.
		From Lemma 2.2($\mathbf{iii}$) of \cite{Hsu.Zhao}, we know that system \eqref{main1.1}-\eqref{main1.3} has at least one positive $T$-periodic solution. Next, we prove system \eqref{main1.1}-\eqref{main1.3} has a unique  stable positive $T$-periodic solution $(\bar{u}(t), \bar{v}(t))$ originating from $(\bar{x}, \bar{y})$.
		
		We first prove the uniqueness of   positive $T$-periodic solution.	Assume that system \eqref{main1.1}-\eqref{main1.3} has another $T$-periodic solution $(\tilde{u}(t), \tilde{v}(t))$ originating from $(\tilde{x}, \tilde{y})$, with
		\begin{equation*}
			\tilde{u}(t) = \begin{cases}
				\tilde{u}_1(t), &\;\; t \in (0, \tau_1], \\
				\tilde{u}_2(t - \tau_1), &\;\; t \in (\tau_1, \tau_2], \\
				\tilde{u}_3(t - \tau_2), &\;\; t \in (\tau_2, T],
			\end{cases}
		\end{equation*}
		and
		\begin{equation*}
			\tilde{v}(t) = \begin{cases}
				\tilde{v}_1(t), &\;\; t \in (0, \tau_1], \\
				\tilde{v}_2(t - \tau_1), &\;\; t \in (\tau_1, \tau_2], \\
				\tilde{v}_3(t - \tau_2), &\;\; t \in (\tau_2, T].
			\end{cases}
		\end{equation*}
		\noindent  Without loss of generality, we assume that $\bar{u}(t) < \tilde{u}(t)$.	
		Note that $(\bar{u}, \bar{v})$ satisfying
		\begin{equation*}
			\begin{cases}
				\displaystyle \frac{d\bar{u}}{dt} = -d_1 \bar{u}, &\;\; t \in (0, \tau_1],\\
				\displaystyle \frac{d\bar{v}}{dt} = -d_2 \bar{v}, &\;\; t \in (0, \tau_1],
			\end{cases}
		\end{equation*}
		\begin{equation*}
			\begin{cases}
				\displaystyle \frac{d\bar{u}}{dt} = \bar{u}(1-\bar{u}-b_1\bar{v}), &\;\; t \in (\tau_1,\tau_2],\\
				\displaystyle \frac{d\bar{v}}{dt} = r \bar{v}(1-\bar{v}-b_2\bar{u}), &\;\; t \in (\tau_1,\tau_2],
			\end{cases}
		\end{equation*}
		\begin{equation*}
			\begin{cases}
				\displaystyle \frac{d\bar{u}}{dt} = \bar{u}(1-\bar{u}-b_1\bar{v}) -c_1 \bar{u}, &\;\; t \in (\tau_2,T],\\
				\displaystyle \frac{d\bar{v}}{dt} = r \bar{v}(1-\bar{v}-b_2\bar{u}) -c_2 \bar{v}, &\;\; t \in (\tau_2,T],
			\end{cases}
		\end{equation*}
		with $(\bar{u}(T), \bar{v}(T)) = (\bar{x}, \bar{y})$.
		We also have
		\begin{equation}\label{eq.pos.1}
			\begin{cases}
				\displaystyle \frac{d\tilde{u}}{dt} = -d_1 \tilde{u}, &\;\; t \in (0, \tau_1],\\
				\displaystyle \frac{d\tilde{v}}{dt} = -d_2 \tilde{v}, &\;\; t \in (0, \tau_1],
			\end{cases}
		\end{equation}
		\begin{equation}\label{eq.pos.2}
			\begin{cases}
				\displaystyle \frac{d\tilde{u}}{dt} = \tilde{u}(1-\tilde{u}-b_1\tilde{v}), &\;\; t \in (\tau_1,\tau_2],\\
				\displaystyle \frac{d\tilde{v}}{dt} = r \tilde{v}(1-\tilde{v}-b_2\tilde{u}), &\;\; t \in (\tau_1,\tau_2],
			\end{cases}
		\end{equation}
		\noindent  and
		\begin{equation}\label{eq.pos.3}
			\begin{cases}
				\displaystyle \frac{d\tilde{u}}{dt} = \tilde{u}(1-\tilde{u}-b_1\tilde{v}) -c_1 \tilde{u}, &\;\; t \in (\tau_2,T],\\
				\displaystyle \frac{d\tilde{v}}{dt} = r \tilde{v}(1-\tilde{v}-b_2\tilde{u}) -c_2 \tilde{v}, &\;\; t \in (\tau_2,T],
			\end{cases}
		\end{equation}
		with $(\tilde{u}(T), \tilde{v}(T)) = (\tilde{x}, \tilde{y}).$
		
		Integrate equation \eqref{eq.pos.1} from 0 to $\tau_1$, equation \eqref{eq.pos.2} from 0 to $\tau_2-\tau_1$, and equation \eqref{eq.pos.3} from 0 to $T-\tau_2$, and then summing them up will yield the result:
		\begin{equation}\label{eq.pos.4}
			-d_1 \tau_1+(\tau_2-\tau_1)+(1-c_1)(T-\tau_2)=\int_0^{\tau_2-\tau_1}(\tilde{u}_2(t)+b_1\tilde{v}_2(t)) dt + \int_0^{T-\tau_2}(\tilde{u}_3(t)+b_1\tilde{v}_3(t)) dt,
		\end{equation}
		and
		\begin{equation}\label{eq.pos.5}
			-d_2 \tau_1+r(\tau_2-\tau_1)+(r-c_2)(T-\tau_2)=\int_0^{\tau_2-\tau_1}r(\tilde{v}_2(t)+b_2\tilde{u}_2(t)) dt + \int_0^{T-\tau_2}r(\tilde{v}_3(t)+b_2\tilde{u}_3(t)) dt.
		\end{equation}
		Multiply formula \eqref{eq.pos.4} by $r$ and subtract formula \eqref{eq.pos.5} multiplied by $b_1$ to obtain the result:
		\begin{equation*}
			\int_0^{\tau_2-\tau_1} (\tilde{u}_2(t)+\tilde{u}_3(t)) dt = \frac{(d_2-d_1)\tau_1+(1-r)(T-\tau_2)+(1-c_1-r+c_2)(T-\tau_2)}{r(1-b_1b_2)}.
		\end{equation*}
		Similarly, we have
		\begin{equation*}
			\int_0^{\tau_2-\tau_1} (\bar{u}_2(t)+\bar{u}_3(t)) dt = \frac{(d_2-d_1)\tau_1+(1-r)(T-\tau_2)+(1-c_1-r+c_2)(T-\tau_2)}{r(1-b_1b_2)},
		\end{equation*}
		which implies that
		\begin{equation*}
			\int_0^{\tau_2-\tau_1} (\tilde{u}_2(t)+\tilde{u}_3(t)) dt = \int_0^{\tau_2-\tau_1} (\bar{u}_2(t)+\bar{u}_3(t)) dt.
		\end{equation*}
		This is obviously contradictory, which indicates that there exists unique positive periodic solution.
		
		We now demonstrate that the unique T-periodic solution $(\bar{u}(t),\bar{v}(t))$ originating from $\bar{w}:=(\bar{x},\bar{y})$ is stable.	
		Let $ V_1(t,\bar{w}),V_2(t,\bar{w}),V_3(t,\bar{w}),f_1,f_2,f_3$ be defined as in the proof of Lemma \ref{lemma4.1}. It follows from the chain rule that
		\begin{equation*}
			DP(\bar{w}) = DQ_{3(T-\tau_2)}(DQ_{2(\tau_2-\tau_1)}(Q_{1\tau_1}(\bar{w}))) \cdot DQ_{2(\tau_2-\tau_1)}(Q_{1\tau_1(\bar{w})})\cdot DQ_{1\tau_1}(\bar{w}).
		\end{equation*}
		\noindent  Then
		\begin{equation*}
			\begin{split}
				DQ_{1\tau_1}(\bar{w}) &= V_1(\tau_1, \bar{w}), \\
				DQ_{2(\tau_2-\tau_1)}(Q_{1\tau_1(\bar{w})}) &= V_2(\tau_2-\tau_1, Q_{1\tau_1}(\bar{w})), \\
				DQ_{3(T-\tau_2)}(DQ_{2(\tau_2-\tau_1)}(Q_{1\tau_1}(\bar{w}))) &= V_3(T-\tau_2, Q_{2(\tau_2-\tau_1)}, Q_{1\tau_1}(\bar{w}))).
			\end{split}
		\end{equation*}
		Furthermore,
		\begin{equation*}
			\begin{split}
				\frac{dV_1(t)}{dt} &= Df_1(U_1(t, \bar{w}))V_1(t), \quad V_1(0) = I, \\
				\frac{dV_2(t)}{dt} &= Df_2(U_2(t, \bar{w}))V_2(t), \quad V_2(0) = I, \\
				\frac{dV_3(t)}{dt} &= Df_3(U_3(t, \bar{w}))V_3(t), \quad V_3(0) = I.
			\end{split}
		\end{equation*}
		\noindent Let
		\begin{equation*}
			h_1(t) = \begin{bmatrix} \frac{1}{\overline{u}_2(t)} & 0 \\ 0 & \frac{1}{\overline{v}_2(t)} \end{bmatrix}.
		\end{equation*}
		Then we obtain
		\begin{equation*}
			h_1(0) = \begin{bmatrix} \frac{1}{\overline{u}e^{-d_1\tau_1}} & 0 \\ 0 & \frac{1}{\overline{v}e^{-d_2\tau_1}} \end{bmatrix}, h_1(\tau_2-\tau_1) = \begin{bmatrix} \frac{1}{\widetilde{u}} & 0 \\ 0 & \frac{1}{\widetilde{v}} \end{bmatrix}.
		\end{equation*}
		\noindent
		
		\noindent Let
		\begin{equation*}
			h_2(t) = \begin{bmatrix} \frac{1}{\overline{u}_3(t)} & 0 \\ 0 & \frac{1}{\overline{v}_3(t)} \end{bmatrix}.
		\end{equation*}
		Then we have
		\begin{equation*}
			h_2(0) = \begin{bmatrix} \frac{1}{\widetilde{u}} & 0 \\ 0 & \frac{1}{\widetilde{v}} \end{bmatrix}, h_2(T-\tau_2) = \begin{bmatrix} \frac{1}{\overline{u}} & 0 \\ 0 & \frac{1}{\overline{v}} \end{bmatrix}.
		\end{equation*}
		\noindent
		
		Let $\psi_1(t)=h_1(t)V_2(t)h_1^{-1}(0)$ , and hence $V_2(t)=h_1^{-1}(t)\psi_1(t)h_1(0)$. Then, we have
		\begin{equation*}
			\frac{d\psi_1(t)}{dt} = \begin{bmatrix} -\overline{u}_2 & -b_1\overline{v}_2 \\ -rd_2\overline{u}_2 & -r\overline{v}_2 \end{bmatrix} \psi_1(t)
		\end{equation*}
		with $\psi_1(0)=I$.
		Let $\psi_2(t)=h_2(t)V_3(t)h_2^{-1}(0)$ , and hence $V_3(t)=h_2^{-1}(t)\psi_2(t)h_2(0)$. Then, we have
		\begin{equation*}
			\frac{d\psi_2(t)}{dt} = \begin{bmatrix} -\overline{u}_3 & -b_1\overline{v}_3 \\ -rd_2\overline{u}_3 & -r\overline{v}_3 \end{bmatrix} \psi_2(t)
		\end{equation*}
		with $\psi_2(0)=I$.
		Hence,
		\begin{eqnarray*}
			DP(\bar{\omega})&=&V_3(T-\tau_2)V_2(\tau_2-\tau_1) \begin{bmatrix} e^{-d_1\tau_1} & 0 \\0 & e^{-d_2\tau_2} \end{bmatrix}  \\
			&=&h_2^{-1}(T-\tau_2) \psi_2(T-\tau_2) h_2(0) h_1^{-1}(\tau_2-\tau_1)\psi_1(\tau_2-\tau_1)h_1(0)\begin{bmatrix} e^{-d_1\tau_1} & 0 \\0 & e^{-d_2\tau_2} \end{bmatrix} \\
			&=&\begin{bmatrix} \bar{u} & 0 \\ 0 & \bar{v} \end{bmatrix} \psi_2(T-\tau_2) \psi_1(\tau_2-\tau_1) \begin{bmatrix}\frac{1}{\bar{u}} & 0 \\0 & \frac{1}{\bar{v}} \end{bmatrix},
		\end{eqnarray*}
		\noindent which implies that $DP(\bar{\omega})$ is similar to $\psi_2(T-\tau_2) \psi_1(\tau_2-\tau_1)$. Thus, we have $ r(DP(\bar{\omega})) = r[\psi_2(T-\tau_2) \psi_1(\tau_2-\tau_1)] $.
		
		Let
		\begin{equation*}
			\xi(t) = \begin{bmatrix} 1 & 0 \\ 0 & -1 \end{bmatrix} \psi_1(t) \begin{bmatrix} 1 & 0 \\ 0 & -1 \end{bmatrix}.
		\end{equation*}
		Then $\xi(t)$ satisfies
		\begin{equation*}
			\begin{cases}
				\displaystyle \frac{d\xi(t)}{dt} = \begin{bmatrix} -\bar{u}_2 & b_1 \bar{v}_2 \\ r d_2 \bar{u}_2 & -r \bar{v}_2 \end{bmatrix} \xi := A(t) \xi, \\
				\displaystyle \xi(0) = I.
			\end{cases}
		\end{equation*}
		Let $\lambda_1$ and $\lambda_2$ be two simple eigenvalues of $r(\xi(\tau_2-\tau_1))$. Since $b_1d_2-1<0$, by Liouville's formula, we have $0 < \lambda_1 \lambda_2 < e^{\int_0^{\tau_2-\tau_1} trace(A(s)) ds} < 1$. Without loss of generality, we assume $0 < \lambda_1 < \lambda_2$. Let $v = (v_1, v_2)^T$ be the positive eigenvector of $\lambda_2$ and $\hat{\xi}(t) = (\hat{\xi}_1(t), \hat{\xi}_2(t))^T := \xi(t)(v_1, v_2)^T$. Since
		\begin{equation*}
			\begin{cases}
				\displaystyle \frac{d\hat{\xi}_1}{dt} = -\bar{u}_2 \hat{\xi}_1 + b_1 \bar{v}_2 \hat{\xi}_2, \\
				\displaystyle \frac{d\hat{\xi}_2}{dt} = r d_2 \bar{u}_2 \hat{\xi}_1 - r \bar{v}_2 \hat{\xi}_2,
			\end{cases}
		\end{equation*}	
		then
		\begin{equation*}
			r \frac{d\hat{\xi}_1}{dt} + b_1 \frac{d\hat{\xi}_2}{dt} = -r\bar{u}_2 \hat{\xi}_1 + r b_1 d_2 \bar{u}_2 \hat{\xi}_1.
		\end{equation*}
		Integrating the above equation for $t$ from $0$ to $\tau_2-\tau_1$, we then obtain
		\begin{equation*}
			(r v_1 + b_1 v_2)(\lambda_2 - 1) = r(b_1d_2-1)\int_0^{\tau_2-\tau_1} \bar{u}_2 \hat{\xi}_1.
		\end{equation*}
		Since condition $(\mathbf{iii})$ implies that $b_1b_2-1<0$, we have $\lambda_2<1$.	
		Thus, we have $r(\psi_1(\tau_2-\tau_1))<1$. Similarly, we can obtain $ r(\psi_2(T-\tau_2))<1$. Then we know that the unique positive periodic solution is stable.
		
		Finally, note that all the boundary steady states, including the trivial steady state  $E_{0}$, two semi-trivial $T$-periodic solutions $(u^*(t),0)$ and $(0,v^*(t))$, are unstable. From Lemma 2.2($\mathbf{iii}$) of \cite{Hsu.Zhao}, we know that every solution of \eqref{main1.1}-\eqref{main1.3} goes to $(\overline{u}(t),\overline{v}(t))$, which indicates that system \eqref{main1.1}-\eqref{main1.3} has a unique globally asymptotically stable positive $T$-periodic solution. This completes the proof.
		\end{proof}

		\subsection{Bistability of semitrivial fixed points}	
		
		\begin{theorem}\label{theorem4.5}
		Assume that the following conditions  are satisfied: \\
		$(\mathbf{i})$ $\tau_1 < \text{min} \{\tau^{*}_1,\tau^{**}_1\}$.\\
		$(\mathbf{ii})$  $\tau_2> \text{max} \{\tau_2^*,\tau_2^{**}\}$.\\
		$(\mathbf{iii})$ $\frac{r c_1}{b_1 c_2} < \frac{\tau_2 - \tau_2^{**}}{\tau_2 - \tau_2^*} < \frac{r b_2 c_1}{c_2}.$ \\Then both the semi-trivial $T$-periodic solution $(u^*(t), 0)$ and $( 0, v^*(t))$ of system \eqref{main1.1}-\eqref{main1.3} are stable.
		\end{theorem}
		\begin{proof}		
		From the conditions (i)-(iii) above and Theorem \ref{lemma4.1}, we see that matrix $DP(\omega_1)$ has two positive eigenvalues $\lambda_1<1$ and $\lambda_2<1$
		and matrix $DP(\omega_2)$ has two positive eigenvalues $\lambda_3<1$ and $\lambda_4<1$.
		Thus, we know that both boundary $T$-periodic solutions $(u^*(t),0)$ and $(0,v^*(t))$ are stable. The proof is completed.
		\end{proof}

		\section{Numerical simulations}
		\ \  \ \
		In this section, we  give numerical examples to demonstrate our analytic results for system \eqref{main1.1}-\eqref{main1.3}.
		
	\subsection{	Stability of semi-trivial periodic solutions}
	
		\noindent $(\mathbf{a})$ Given the parameters
		\begin{align*}
			d_1 &= 0.4, & d_2 &= 0.8, & b_1 &=0.2, & b_2 &= 0.2, & r &= 1, \\
			c_1 &= 0.4, & c_2 &= 0.6, & \tau_1 &= 6, & \tau_2 &= 9, & T &=12,
		\end{align*}
		the thresholds are $\tau_1^{*}=8.57$, $\tau_1^{**}=6.67$, $\tau_2^{*}=3$ and $\tau_2^{**}=10$, which implies that $\tau_1<\tau^{*}_1$ and $\tau_2^*<\tau_2<\tau_2^{**}$. It follows from Theorem \ref{theorem4.2} that the unique semi-trivial $T$-periodic solution $(u^*(t),0)$ of system \eqref{main1.1}-\eqref{main1.3} is globally asymptotically stable as shown in Figure \ref{example2a}.
		$(\mathbf{b})$ Given the parameters
		\begin{align*}
			d_1 &= 0.8, & d_2 &= 0.4, & b_1 &=0.2, & b_2 &= 0.2, & r &= 1, \\
			c_1 &= 0.6, & c_2 &= 0.4, & \tau_1 &= 6, & \tau_2 &= 9, & T &=12,
		\end{align*}
		the thresholds are $\tau_1^{*}=6.67$, $\tau_1^{**}=8.57$, $\tau_2^{*}=10$ and $\tau_2^{**}=3$, which implies that $\tau_1<\tau^{**}_1$ and $\tau_2^{**}<\tau_2<\tau_2^{*}$. It follows from Theorem \ref{theorem4.3}  that the unique semi-trivial $T$-periodic solution $(0,v^*(t))$ of system \eqref{main1.1}-\eqref{main1.3} is globally asymptotically stable as shown in Figure \ref{example2b}.

		\begin{figure}[htbp]
		\begin{subfigure}[b]{0.49\textwidth}
			\centering
			\includegraphics[width=1.05\textwidth]{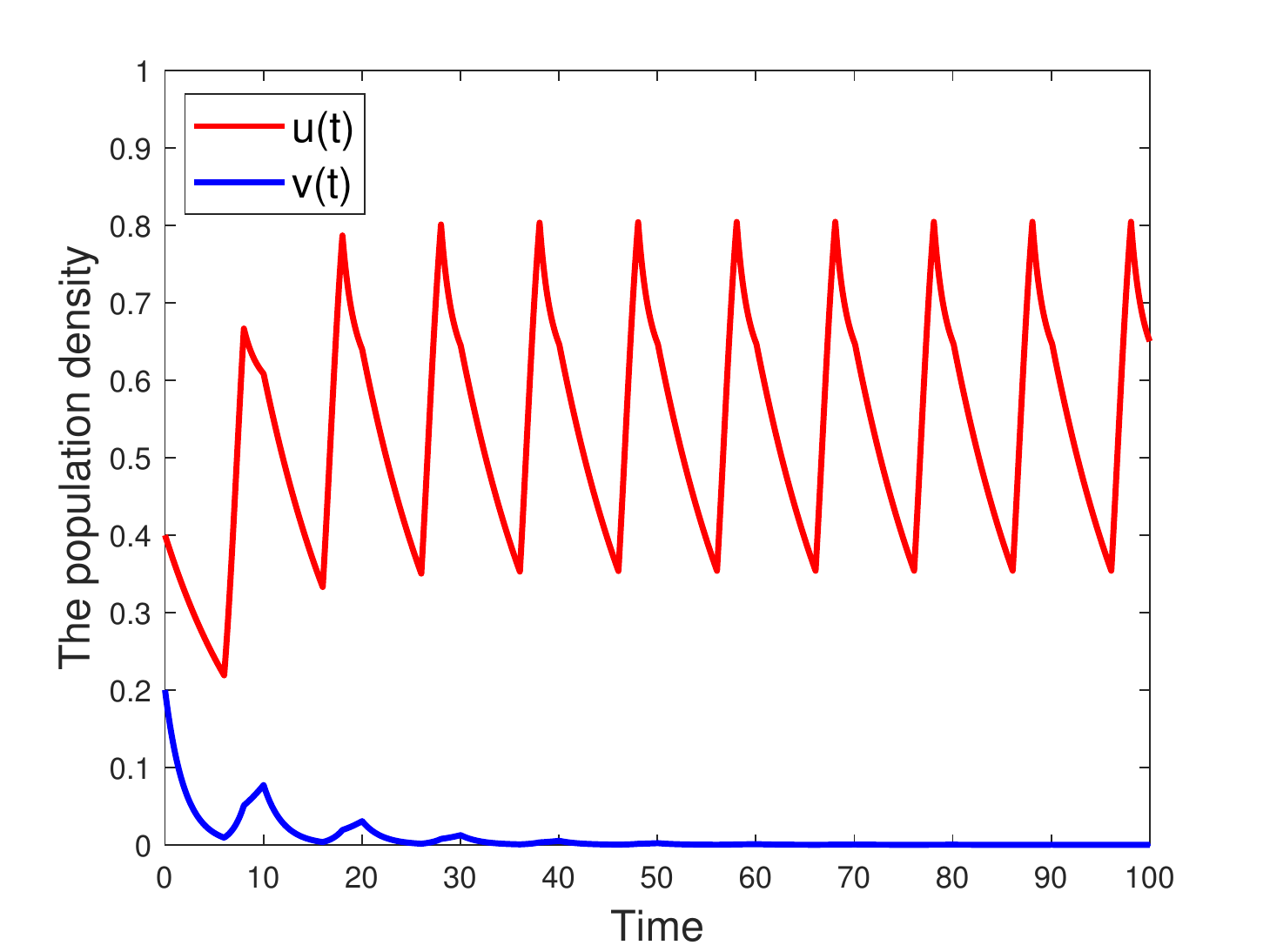}
			\caption{}
			\label{example2a}
		\end{subfigure}
		\hfill
		\begin{subfigure}[b]{0.49\textwidth}
			\centering
			\includegraphics[width=1.05\textwidth]{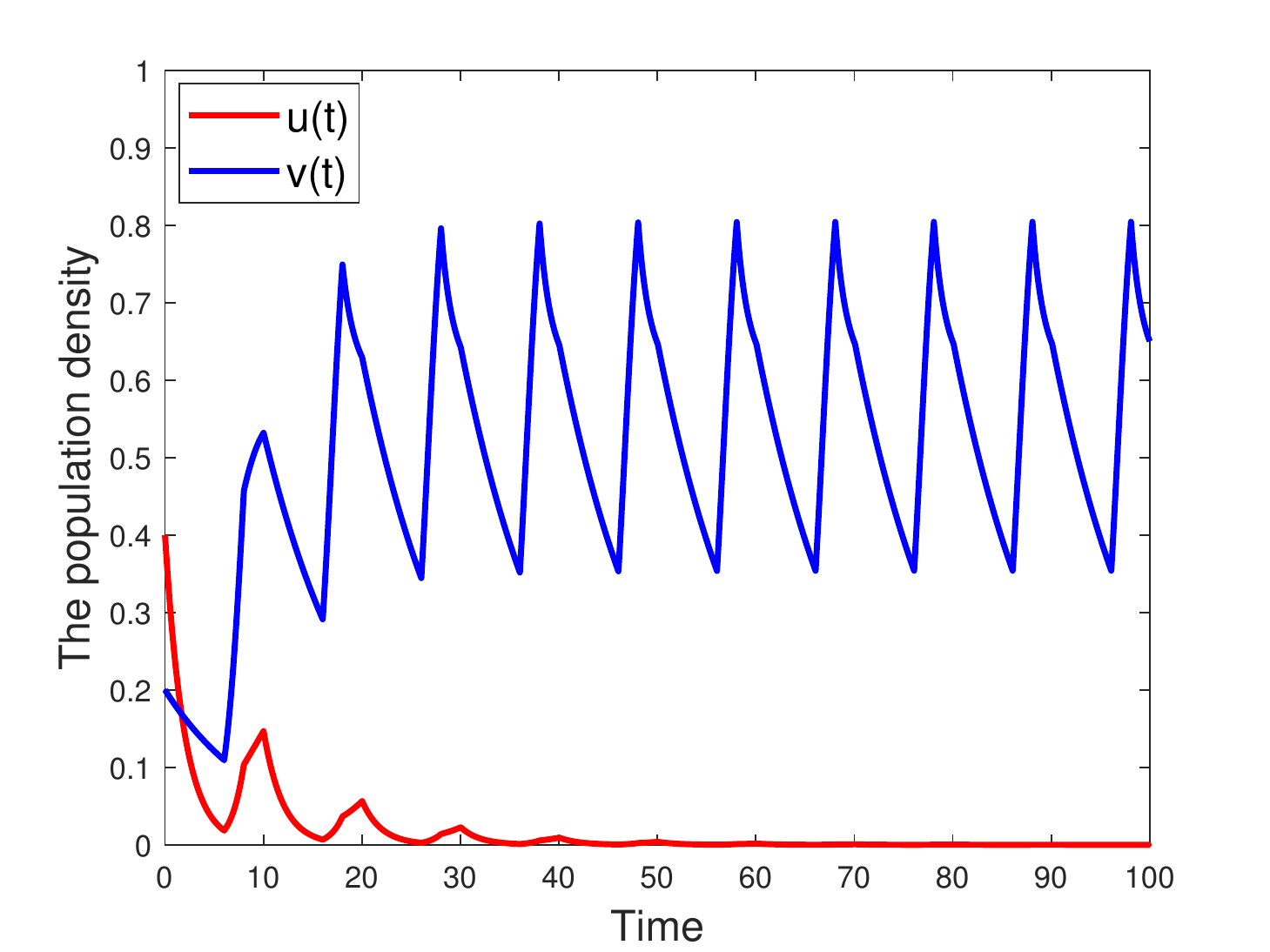}
			\caption{}
			\label{example2b}
		\end{subfigure}
		\captionsetup{font={footnotesize}}
		\caption{Stability of two semi-trivial periodic solutions $(u^*(t),0)$ and $(0,v^*(t))$. (a) Semi-trivial periodic solution $(u^*(t),0)$ is stable. (b) $(0,v^*(t))$ is stable. }
		\end{figure}

		\subsection{Coexistence under weak competition}

		Given the parameters
		\begin{align*}
			d_1 &= 0.6, & d_2 &= 0.4, & b_1 &=0.2, & b_2 &= 0.2, & r &= 1, \\
			c_1 &= 0.6, & c_2 &= 0.4, & \tau_1 &= 6, & \tau_2 &= 9, & T &=12,
		\end{align*}
	which implies that there exists weak competition in system.	The thresholds are $\tau_1^{*}=7.5$, $\tau_1^{**}=8.57$, $\tau_2^{*}=8$ and $\tau_2^{**}=3$, which implies that conditions ($\mathbf{i}$)-($\mathbf{iii}$) in Theorem \ref{theorem4.4} are satisfied. It follows from Theorem \ref{theorem4.4} that system \eqref{main1.1}-\eqref{main1.3} has a unique globally asymptotically stable positive periodic solution $(\bar{u}(t),\bar{v}(t))$ originating from $(\bar{u},\bar{v})$ as shown in Figure \ref{example3}.

		\begin{figure}[htbp]
		\centering
		\includegraphics[width=0.55\textwidth]{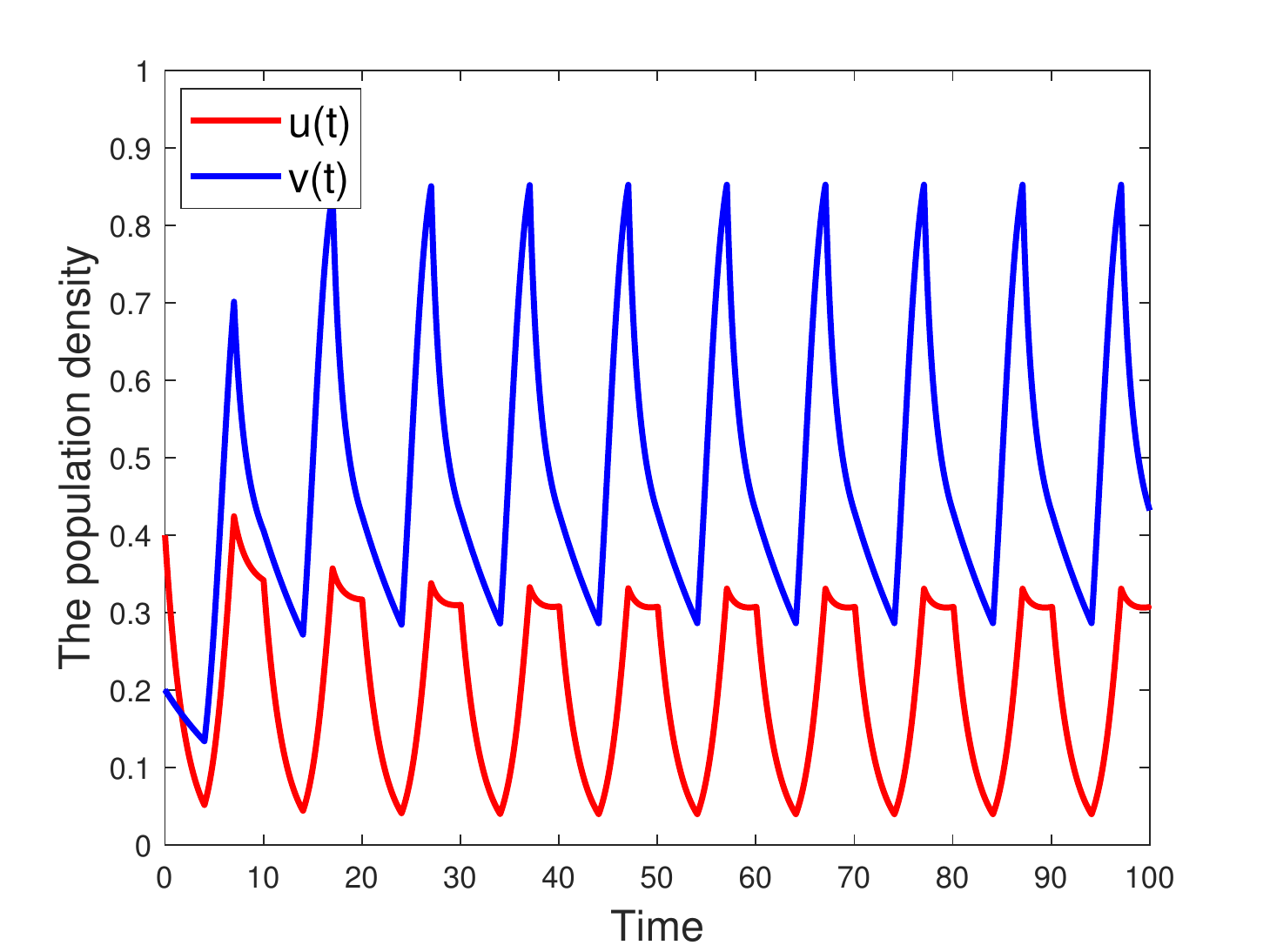}
		\captionsetup{font={footnotesize}}
		\caption{Periodic coexistence of two species $u(t)$ and $v(t)$. There exists a globally asymptotically stable positive periodic solution.}
		\label{example3}
		\end{figure}	
		
		\subsection{Bistability under strong competition}

		Given the parameters
		\begin{align*}
			d_1 &= 0.4, & d_2 &= 0.4, & b_1 &=2, & b_2 &= 2, & r &= 1, \\
			c_1 &= 0.6, & c_2 &= 0.4, & \tau_1 &= 6, & \tau_2 &= 9, & T &=12,
		\end{align*}
		which implies that there exists strong competition in system.  The thresholds are $\tau_1^{*}=8.57$,  $\tau_1^{**}=8.57$, $\tau_2^{*}=6$ and $\tau_2^{**}=3$. It follows from Theorem \ref{theorem4.5} that there exists  bistability of boundary equilibrium point in system \eqref{main1.1}-\eqref{main1.3} as shown in Figures \ref{example4a} and \ref{example4b}.

		\begin{figure}[htbp]
		\begin{subfigure}[b]{0.49\textwidth}
			\centering
			\includegraphics[width=1.05\textwidth]{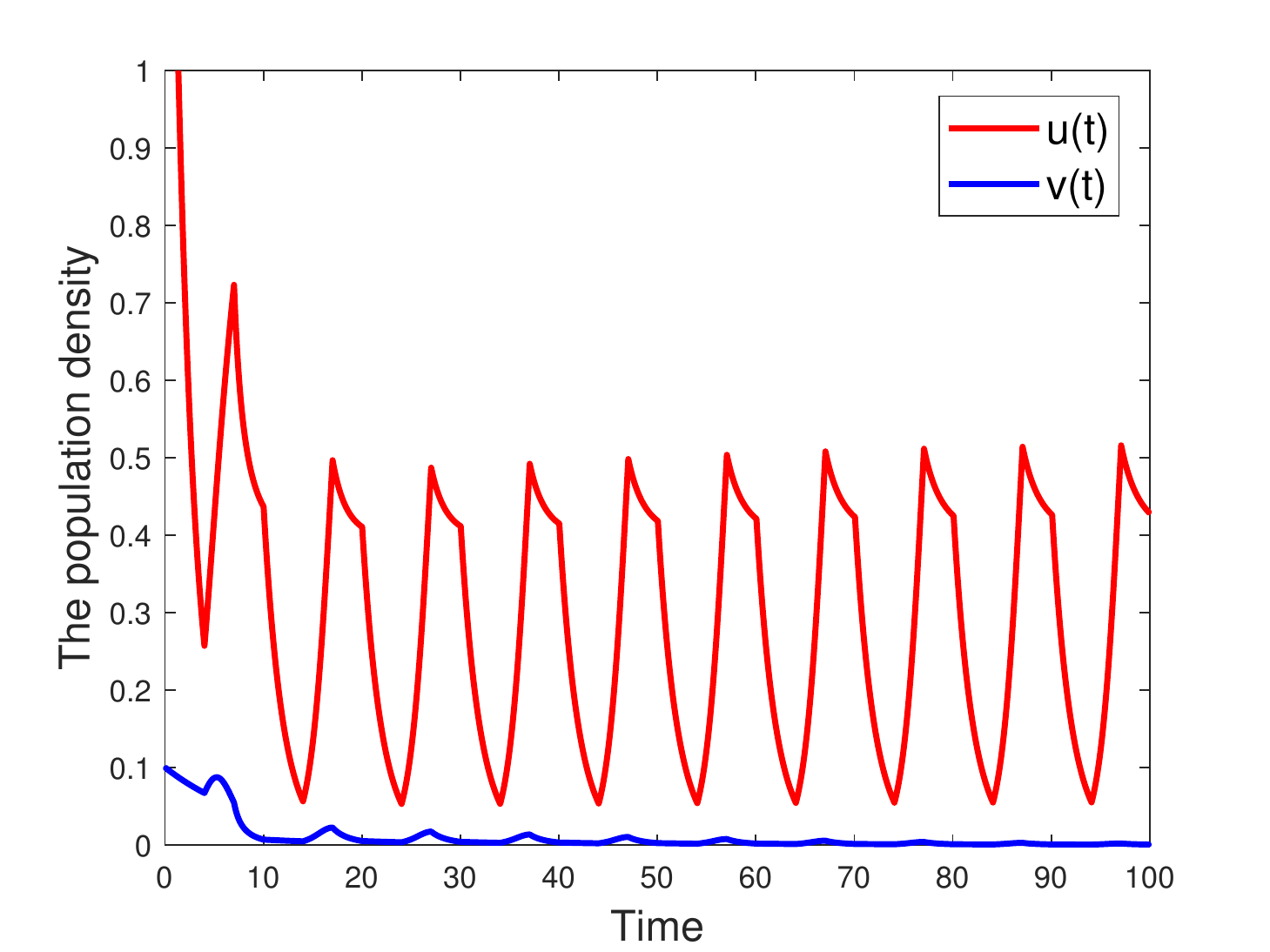}
			\caption{}
			\label{example4a}
		\end{subfigure}
		\hfill
		\begin{subfigure}[b]{0.49\textwidth}
			\centering
			\includegraphics[width=1.05\textwidth]{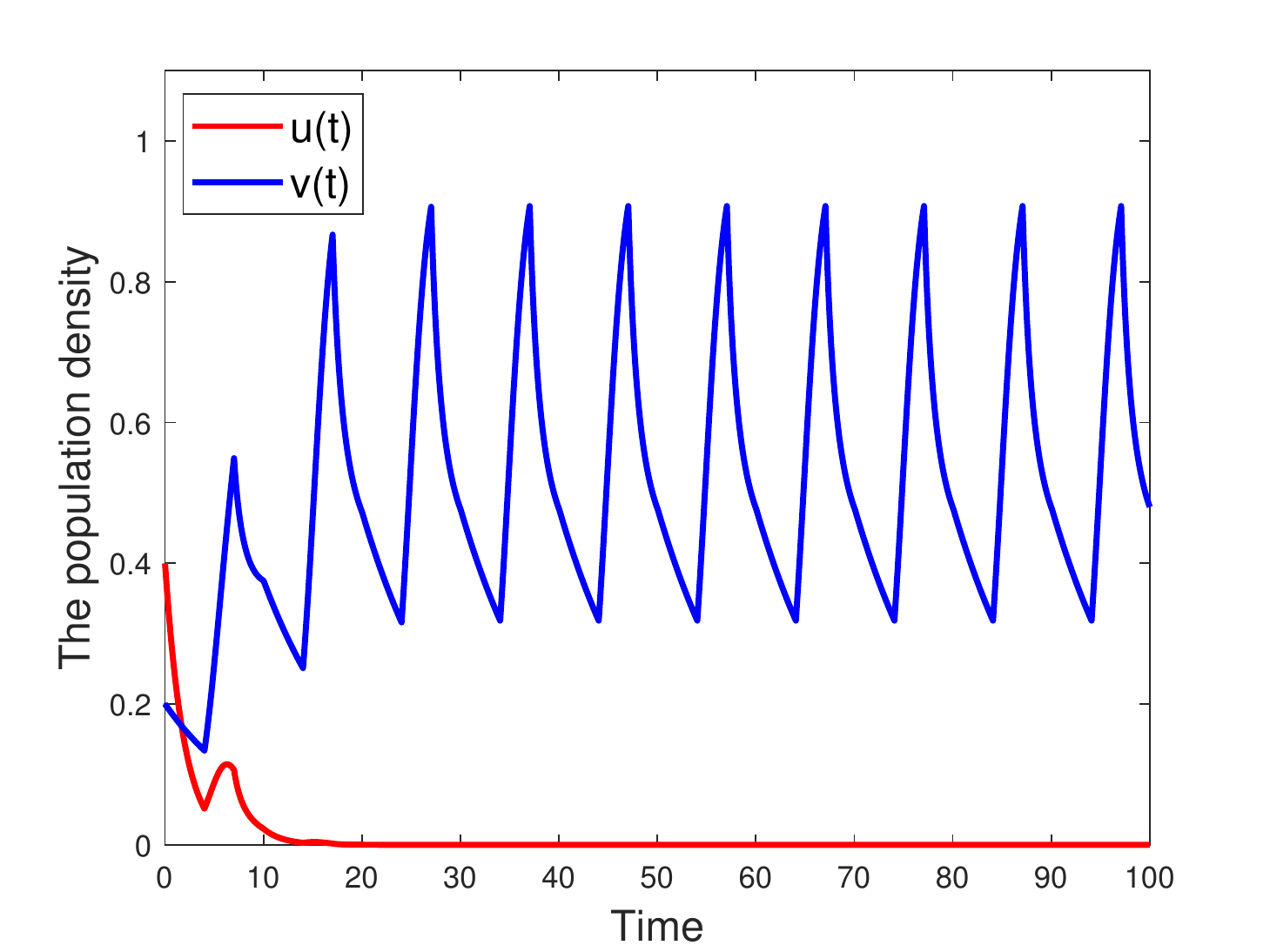}
			\caption{}
			\label{example4b}
		\end{subfigure}
		\captionsetup{font={footnotesize}}
		\caption{Bistability of two semi-trivial periodic solution $(u^*(t),0)$ and $(0,v^*(t))$. (a). The initial value is $u_0=2$, $v_0=0.1$. (b). The initial value is $u_0=0.4$, and $v_0=0.2$. }
		\end{figure}
		
		\section{Conclusion and discussion }
		\ \ \ \
		The duration and intensity of grazing management are crucial for the sustainable utilization of grassland resources in arid and semi-arid regions. This study explores effective grazing management strategies to ensure the long-term stability of these resources and ecosystem health.
		
		We first analyze the global dynamics of a single vegetation population under seasonal grazing (see Theorems \ref{lem.u.solu}-\ref{lem.v.solu}, Theorems \ref{lem.u.tao1}-\ref{lem.v0}). We established precise critical durations for the dry season
  and grazing period that determine whether a vegetation population persists or goes extinct. That is $\tau_1^{*}=\frac{T}{d_1+1}$, $\tau_2^*=\frac{(d_1+1)\tau_1+(c_1-1)T}{c_1}$, $\tau_1^{**}=\frac{rT}{d_2+r}$, and $\tau_2^{**}=\frac{(d_2+r)\tau_1+(c_2-r)T}{c_2}$. The bifurcation diagrams on the plane $\tau_1-\tau_2$  reveal that
  under low grazing intensity ($c_1 < 1$ or $c_2 < r$), vegetation can persist even if grazing occurs throughout the growth period (see Figure \ref{fig.u&v}).
Under high grazing intensity ($c_1 > 1$ or $c_2 >r$), persistence requires that grazing be restricted to only part of the growth season.

	These thresholds provide actionable grazing strategies for grazing management. For instance, for semi-arid regions like northwestern China, when $d_{1}=0.5$, we have $\tau^{\ast}_{1}=12/(d_{1}+1)=8$. If local dry seasons last $6$ months ($\tau_{1}=6, T=12$), we have  $\tau^{\ast}_{1}=6<8$, meaning grazing can be allowed if $\tau_{2}>\tau_{2}^{\ast}=[(0.5+1)6+(c_{1}-1)12]/c_{1}$. For low grazing intensity $c_{1}=0.4$, $\tau^{\ast}_{2}=4.5$, so grazing can start after  dry season. For high grazing intensity ($c_{1}=1.2$), $\tau_{2}^{\ast}=9.5$, requiring grazing to be delayed until 3.5 months of growth to avoid species extinction.
		\begin{figure}[H]
		\centering
		\begin{subfigure}[b]{0.48\textwidth}
			\includegraphics[width=\textwidth]{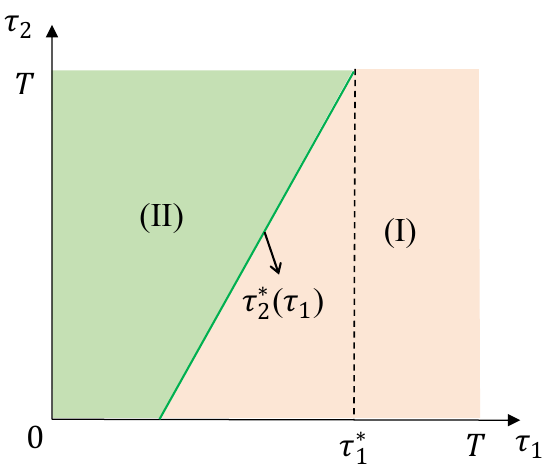}
			\caption{$0<c_1<1$}
			\label{fig.u.1}
		\end{subfigure}
		\hfill
		\begin{subfigure}[b]{0.48\textwidth}
			\includegraphics[width=\textwidth]{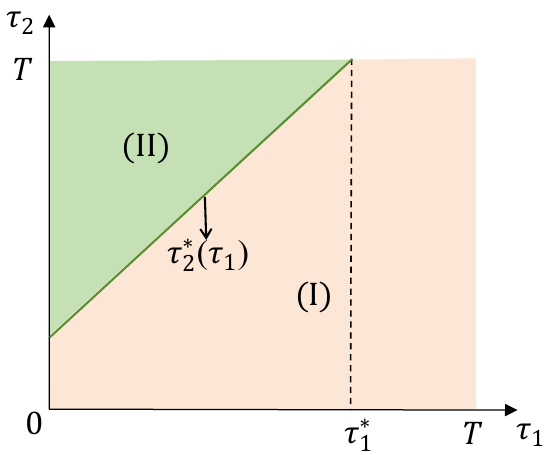}
			\caption{$c_1>1$}
			\label{fig.u.2}
		\end{subfigure}
		
		\vspace{1em} 
		
		\begin{subfigure}[b]{0.48\textwidth}
			\includegraphics[width=\textwidth]{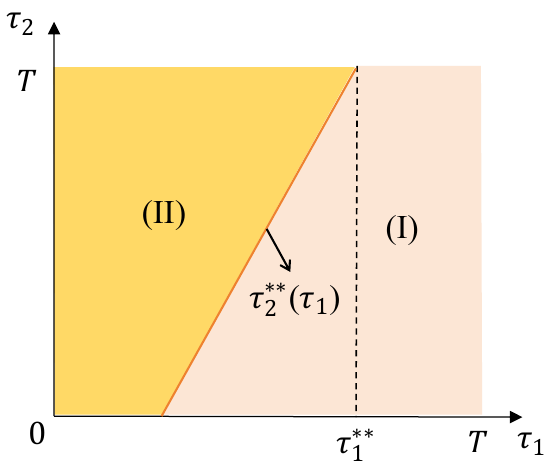}
			\caption{$0<c_2<r$}
			\label{fig.v.1}
		\end{subfigure}
		\hfill
		\begin{subfigure}[b]{0.48\textwidth}
			\includegraphics[width=\textwidth]{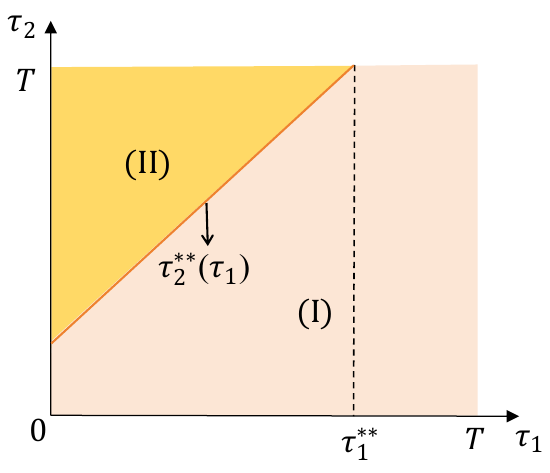}
			\caption{$c_2>r$}
			\label{fig.v.2}
		\end{subfigure}
		\captionsetup{font={footnotesize}}
		\caption{Bifurcation diagrams in $\tau_{1}-\tau_{2}$ plane for single species $u$ (Figures (a) and (b)) and  $v$ (Figures (c) and (d)). In region (I), the trivial solution is GAS. In region (II), unique GAS positive periodic solution exists.   }
		\label{fig.u&v}
		\end{figure}

		We then studied a two-species vegetation competition system subjected to varying grazing periods and intensities. Our research indicates that grazing intensity and duration are crucial for the stability of both trivial and non-trivial solutions. We identified four critical thresholds related to the durations of the dry season and grazing season:$\tau_1^{*}, \tau_2^*, \tau_1^{**}, \tau_2^{**}.$
		These thresholds classify the system into different dynamic regimes.

		(1)Extinction: both species die out if the dry season is too long or grazing is excessive.

(2)Competitive exclusion: one species dominates when its grazing period is shorter than the other's.

(3)Coexistence: a unique globally stable positive periodic solution exists under short grazing and weak competition ($b_1 b_2<1$).

(4)Bistability: both semi-trivial solutions are stable under strong competition ($b_1 b_2>1$).

These results are visualized in bifurcation diagrams (Figures \ref{fig.pos} and \ref{fig.bis}), highlighting how grazing management can alter competitive outcomes.
		In particular, Figure \ref{fig.pos} presents the bifurcation diagram on the $\tau_{1}$-$\tau_{2}$ plane for the weak competition case ($b_1 b_2<1$), illustrating the effects of dry season and grazing period under varying grazing intensities $c_{1}$ and $c_{2}$. One can note that there exist regions supporting a globally asymptotically stable  positive periodic solution. But there is no regions where both boundary solutions $(u^{*}(t),0)$ and $(0,v^{*}(t))$ are simultaneously stable.
		
Figure \ref{fig.bis} shows the bifurcation diagram for the strong competition case ($b_1 b_2>1$). One can see that, conversely,  there exist regions  where both boundary solutions are stable, but no regions support a GAS positive periodic solution.
		We note that the global dynamics remain unclear for certain parameter combinations, such as in regions (V) and (VI) of Figures \ref{fig.pos} and \ref{fig.bis}.

		Our model offers actionable insights for sustainable grassland management in semi-arid regions. By adjusting the timing and intensity of grazing based on seasonal and species-specific thresholds, managers can promote biodiversity, prevent overgrazing, and enhance ecosystem resilience.

While our model captures essential seasonal and grazing effects, several extensions are needed.
Firstly, we can incorporate more realistic harvesting functions (e.g., Michaelis¨CMenten type).
Additionally, it is interesting to introducing spatial heterogeneity to study pattern formation and grazing impacts in fragmented landscapes.
At the same time, exploring unresolved dynamics in regions (V) and (VI) of the bifurcation diagrams are important.
Future work will also consider stochastic environmental variations and multi-species interactions to further bridge theoretical models and field applications.

		\begin{figure}[htbp]
		\centering
		\begin{subfigure}[b]{0.48\textwidth}
			\includegraphics[width=\textwidth]{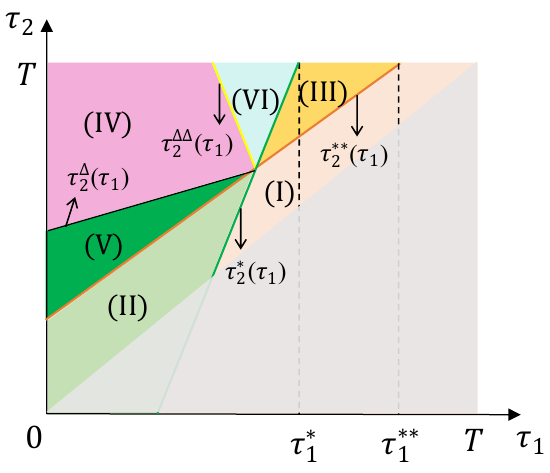}
			\caption{$\tau_1^{*}<\tau_1^{**}, 0<c_1<1, c_2>r$}
		\end{subfigure}
		\hfill
		\begin{subfigure}[b]{0.48\textwidth}
			\includegraphics[width=\textwidth]{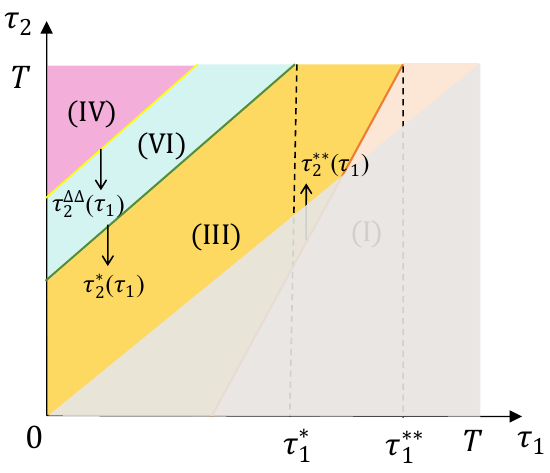}
			\caption{$\tau_1^{*}<\tau_1^{**}, c_1>1, 0<c_2<r$}
		\end{subfigure}
		
		\vspace{1em} 
		
		\begin{subfigure}[b]{0.48\textwidth}
			\includegraphics[width=\textwidth]{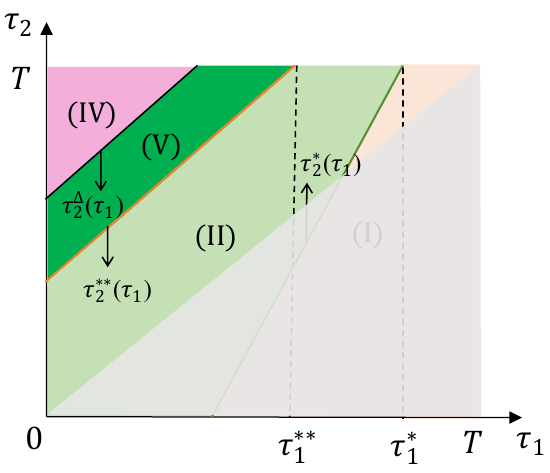}
			\caption{$\tau_1^{**}<\tau_1^{*}, 0<c_1<1, c_2>r$}
		\end{subfigure}
		\hfill
		\begin{subfigure}[b]{0.48\textwidth}
			\includegraphics[width=\textwidth]{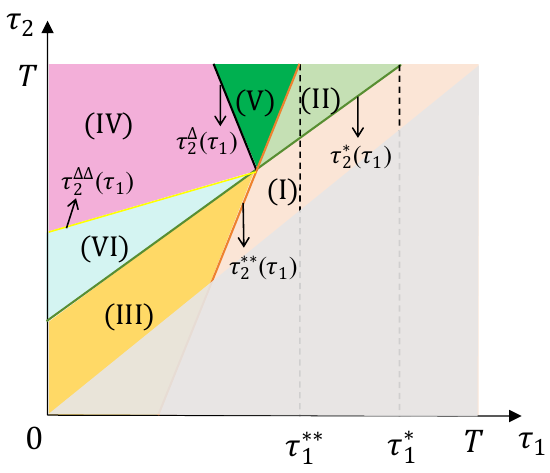}
			\caption{$\tau_1^{**}<\tau_1^{*}, c_1>1, 0<c_2<r$}
		\end{subfigure}
		\captionsetup{font={footnotesize}}
		\caption{Impacts of the duration and  intensity of grazing cycle on  the dynamics of  the competition system under $b_1 b_2<1$. Region (I): the trivial solution $E_0$ is GAS;  Region (II): the semi-trivial solution $(u^*(t),0)$ is GAS;  Region (III): the semi-trivial solution $(0,v^*(t))$ is GAS; Region (IV):  the unique  positive periodic solution is GAS. Region (V): the semi-trivial solution $(u^*(t),0)$ is locally asymptotically stable  (LAS), but it is not confirmed whether there are positive periodic solutions. Region (VI): the semi-trivial solution $(0,v^*(t))$ is LAS. It is not confirmed whether there are positive periodic solutions. The straight line $\tau_2^{\triangle}(\tau_1)$ is derived from the equation  $\frac{\tau_2-\tau_2{**}}{\tau_2-\tau_2^*}=\frac{rb_2c_1}{c_2}$. The  straight line $\tau_2^{\triangle\triangle}(\tau_1)$ is derived from the  equation $\frac{\tau_2-\tau_2{**}}{\tau_2-\tau_2^*}=\frac{rc_1}{b_1c_2}$.}
		\label{fig.pos}
		\end{figure}
		
		\begin{figure}[H]
		\centering
		\begin{subfigure}[b]{0.48\textwidth}
			\includegraphics[width=\textwidth]{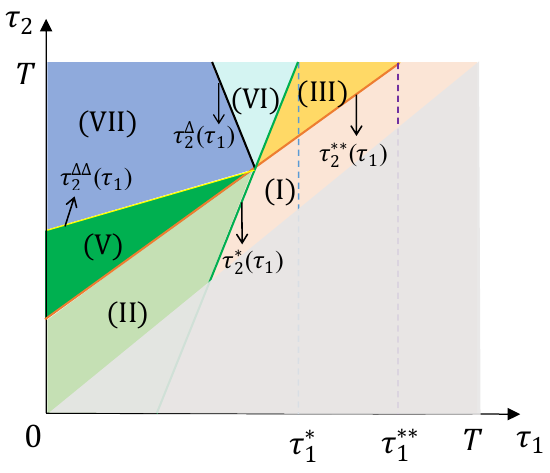}
			\caption{$\tau_1^{*}<\tau_1^{**}, 0<c_1<1, c_2>r$}
		\end{subfigure}
		\hfill
		\begin{subfigure}[b]{0.48\textwidth}
			\includegraphics[width=\textwidth]{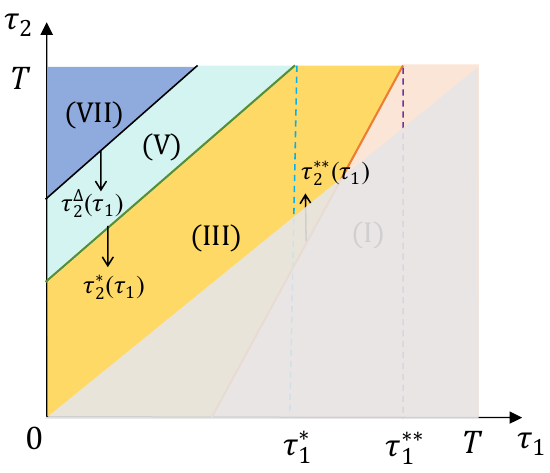}
			\caption{$\tau_1^{*}<\tau_1^{**}, c_1>1, 0<c_2<r$}
		\end{subfigure}
		
		\vspace{1em} 
		
		\begin{subfigure}[b]{0.48\textwidth}
			\includegraphics[width=\textwidth]{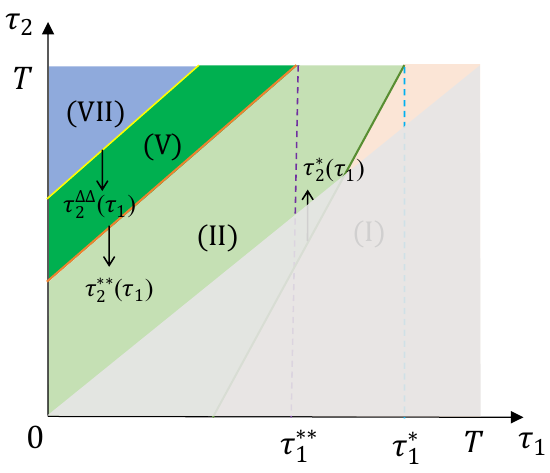}
			\caption{$\tau_1^{**}<\tau_1^{*}, 0<c_1<1, c_2>r$}
		\end{subfigure}
		\hfill
		\begin{subfigure}[b]{0.48\textwidth}
			\includegraphics[width=\textwidth]{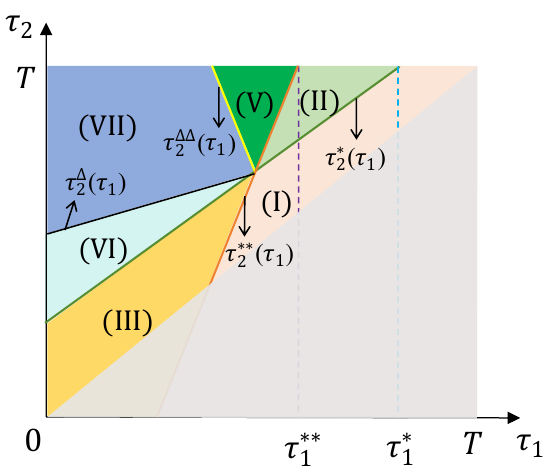}
			\caption{$\tau_1^{**}<\tau_1^{*}, c_1>1, 0<c_2<r$}
		\end{subfigure}
		\captionsetup{font={footnotesize}}
		\caption{Impacts of the duration and  intensity of grazing cycle on  the dynamics of  the competition system under $b_1 b_2>1$. In region (I), the trivial solution $E_0$ is GAS. Region (II) and (III) represent that the semi-trivial solution $(u^*(t),0)$ and the semi-trivial solution $(0,v^*(t))$ are  GAS, respectively. Region (VII) represents that there exists bistable solution $(u^*(t),0)$ and $(0,v^*(t))$. In region (V), the semi-trivial solution $(u^*(t),0)$ is LAS, but it is not confirmed whether there are positive periodic solutions. In region (VI), the semi-trivial solution $(0,v^*(t))$ is LAS, but it is not confirmed whether there are positive periodic solutions. }
		\label{fig.bis}
		\end{figure}


\begin{thebibliography}{}
\bibitem{vegetation1}C. A. Klausmeier, Regular and irregular patterns in semiarid vegetation, Science(1999), 284:1826-1828.
\bibitem{vegetation2} Rietkerk, S. C. Dekker, P. C. de Ruiter, and J. van de Koppel, Self-organized patchiness and catastrophic shifts in ecosystems, Science(2004), 305: 1926-1929.
	\bibitem{vegetation3}
J. von Hardenberg, E. Meron, M. Shachak, Diversity of vegetation patterns
and desertifcation. Phys. Rev. Lett. (2001), 87:198101.



		\bibitem{vegetation4}  J. A. Sherratt, Pattern solutions of the Klausmeier model for banded vegetation in semi-arid
		environments II: Patterns with the largest possible propagation speeds, Proc. R. Soc. A(2011),
		467: 3272-3294.
	\bibitem{vegetation5}	J. A. Sherratt, Pattern solutions of the Klausmeier model for banded vegetation in
semiarid environments V: the transition from patterns to desert. SIAM. J. Appl.
Math. (2013), 73(4):1347-1367.
		
		\bibitem{vegetation6}
		X.-L. Wang, G.-H. Zhang, Vegetation pattern formation in seminal systems due to internal competition reaction between plants, J.  Theor. Biol. (2018), 458:10-14.
	\bibitem{vegetation7}	X.-L. Wang, J.-P. Shi, G.-H. Zhang, Bifurcation and pattern formation in diffusive
Klausmeier-Gray-Scott model of water-plant interaction. J Math Anal Appl
(2021), 497(1):124860.


\bibitem{vegetation11}
		L. Eigentler, J.A. Sherratt, Effects of precipitation intermittency on vegetation patterns in
		semi-arid landscape, Physica D. (2020), 405: 132396.
		\bibitem{vegetation12}L. Eigentler,  J.A. Sherratt, An integrodifference model for vegetation patterns in
		semi-arid environments with seasonality, J. Math. Biol. (2020), 81 (3):875-904.
		\bibitem{vgrazing1}
		M. K. Pal, S. Pori, Effect of nonlocal grazing on dry-land vegetation dynamics, Phys. Review. E. (2022), 106: 054407 .
		
		\bibitem{vgrazing2}
		Y. Maimaiti, W.-B. Yang, Spatial vegetation pattern formation and transition of an
		extended water-plant model with nonlocal or local grazing, Nonlinear Dyn. (2024),  112 5765-5791.
		
		\bibitem{vgrazing3}
		E. Siero, Nonlocal grazing in patterned ecosystem, J.  Theor. Biol. (2018), 436:  64-71.


\bibitem{Harvest1} P. Schonbach, H. Wan, M. Gierus, et al. Grassland responses to grazing: effects
of grazing intensity and management system in an Inner Mongolian steppe
ecosystem. Plant. Soil. (2011), 340(1):103-115.
\bibitem{Harvest2} Y. Lin, M. Hong, G.-D. Han, et al. Grazing intensity affected spatial patterns
of vegetation and soil fertility in a desert steppe. Agric. Ecosyst. Environ.
(2010), 138:282¨C292.
\bibitem{Harvest3} D.J. Eldridge, A.G.B. Poore, M. Ruiz-Colmenero, et al. Ecosystem structure, function,
and composition in rangelands are negatively affected by livestock grazing. Ecol.
Appl. (2016), 26(4):1273-1283.



\bibitem{Harvest5}G.-Q. Sun , H.-T. Zhang, J.S.Wang, et al. Mathematical modeling and mechanisms of pattern formation in ecological systems: a review. Nonlinear Dynam.
(2021), 104(2):1677-1696.
\bibitem{Harvest6} F. Fan, C. Liang, Y. Tang, et al. Effects and relationships of grazing intensity on
multiple ecosystem services in the Inner Mongolian steppe. Sci. Total. Environ.
(2019), 675:642-650.
		
\bibitem{Harvest7}E. Siero, Nonlocal grazing in patterned ecosystems, J.  Theor. Biol. (2018), 436: 64-71.


\bibitem{Harvestsun}J. Li, G.-Q. Sun, L. Li, Z. Jin, Y. Yuan, The effect of grazing intensity on pattern dynamics of the vegetation system, Chaos, Solitons and Fractals, 175 (2023) 114025.


\bibitem{seasonal} K.L Metzger, M.B. Coughenour, R.M. Reich, et al. Effects of seasonal grazing on
plant species diversity and vegetation structure in a semi-arid ecosystem. J. Arid.
Environ. (2005), 61(1):147-160.


\bibitem{Hsu.Zhao}
		S.-B. Hsu, X.-Q. Zhao,
		\newblock A {L}otka-{V}olterra competition model with seasonal succession.
		\newblock {\em J. Math. Biol. (2012)}, 64(1-2):109-130.




		\bibitem{Liu.MMHarvest1}
		X.-M. Feng, Y.-F. Liu, S.-G. Ruan, J.-S. Yu,
		\newblock Periodic dynamics of a single species model with seasonal
		{M}ichaelis-{M}enten type harvesting.
		\newblock {\em J. Differ. Equ. (2023)}, 354:237-263.
		
	
		\bibitem{succession1}
		R. Peng, X.-Q. Zhao, The diffusive logistic model with a free boundary and seasonal succession, Discrete Contin.
		Dyn. Syst. (2013), 33: 2007-2031.
		\bibitem{succession2}
		L. Pu, Z. Lin, Y. Lou, A West Nile virus nonlocal model with free boundaries and seasonal succession, J. Math.
		Biol. (2023), 86:  1-52
		\bibitem{succession3}
		L. Niu, Y. Wang, X. Xie, Carrying simplex in the Lotka-Volterra competition model with seasonal succession with
		applications, Discrete Contin. Dyn. Syst., Ser. B (2021), 26:  2161-2172.
		\bibitem{succession4}
		U. Sommer, Z.M. Gliwicz, W. Lampert, A. Duncan, The PEG-model of seasonal succession of planktonic events in
		fresh waters, Arch. Hydrobiol. (1986), 106:  433-471.
		\bibitem{succession5}
		M. Wang, Q. Zhang, X.-Q. Zhao, Dynamics for a diffusive competition model with seasonal succession and different
		free boundaries, J. Differ. Equ. (2021), 285:  536-582.
		
		
		
		\bibitem{Liu.MMHarvest2}
		Y.-F. Liu, X.-M. Feng, S.-G. Ruan, and J.S. Yu.
		\newblock Periodic dynamics of a single species model with seasonal
		{M}ichaelis-{M}enten type harvesting, {II}: {E}xistence of two periodic
		solutions.
		\newblock {\em J. Differ. Equ. (2024)}, 388:253-285.
		\bibitem{SHU2025}
		X.-J. Pan, H.-Y. Shu, L. Wang, X.-S. Wang, J.-S. Yu
		\newblock On the periodic solutions of switching scalar dynamical
		systems
		\newblock {\em J. Differ. Equ. (2025)},  415:  365-382.
		\bibitem{Liu.2cs}
		Y.-F. Liu, J.-S. Yu, and J. Li.
		\newblock Global dynamics of a competitive system with seasonal succession and
		different harvesting strategies.
		\newblock {\em J. Differ. Equ. (2024)}, 382:211-245.
		
		



		
		
		
		\bibitem{Yu.mosquito1}
		J.-S. Yu, J. Li,
		\newblock Global asymptotic stability in an interactive wild and sterile
		mosquito model.
		\newblock {\em J. Differ. Equ. (2020)}, 269(7):6193-6215.
		
		\bibitem{Yu.mosquito2}
		J.-S. Yu,
		\newblock Existence and stability of a unique and exact two periodic orbits for an interactive wild and sterile mosquito
		model.
		\newblock {\em J. Differ. Equ. (2020)}, 269:  10395-10415.
		
		
		\bibitem{Yu.mosquito3}
		B. Zheng, J. Li, J. Yu,
		\newblock One discrete dynamical model on Wolbachia infection frequency in mosquito populations.
		\newblock {\em Sci. China Math. (2022)}, 65:  1749-1764.
		\bibitem{Yu.mosquito4}
		B. Zheng, J. Yu, J. Li,
		\newblock Modeling and analysis of the implementation of the Wolbachia incompatible and sterile insect
		technique for mosquito population suppression.
		\newblock {\em SIAM J. Appl. Math. (2021)}, 81:  718-740.
		
		
		
		
		
		\end{thebibliography}
	\end{document}